\documentclass[leqno,12pt]{amsart} 
\usepackage{amssymb}
\theoremstyle{plain} 
\newtheorem{theorem}{\indent\sc Theorem}[section]
\newtheorem{lemma}[theorem]{\indent\sc Lemma}
\newtheorem{corollary}[theorem]{\indent\sc Corollary}

\theoremstyle{definition} 
\newtheorem{definition}[theorem]{\indent\sc Definition}
\newtheorem{remark}[theorem]{\indent\sc Remark}
\newtheorem{example}[theorem]{\indent\sc Example}

\def\C{{\mathbf{C}}}
\def\R{{\mathbf{R}}}
\def\H{{\mathbf{H}}}
\def\N{{\mathbf{N}}}
\def\Pi{{\mathbf{P}}}
\def\Si{{\mathbf{S}}}
\def\Lc{{\mathcal{L}}}

\def\tr#1{\mathord{\mathopen{{\vphantom{#1}}^t}#1}}

\begin{document}

\title[Function-theoretic properties of Gauss maps]{Function-theoretic properties for the Gauss maps of various classes of surfaces} 

\author[Y. Kawakami]{Yu Kawakami} 



\renewcommand{\thefootnote}{\fnsymbol{footnote}}
\footnote[0]{2010\textit{ Mathematics Subject Classification}.
Primary 30D35, 53C42; Secondary 30F45, 53A10, 53A15.}
\thanks{ 
Partly supported by the Grant-in-Aid for Young Scientists (B) No. 24740044, Japan Society for the Promotion of Science.
}
\address{
Graduate School of Natural Science and Technology, \endgraf
Kanazawa university, \endgraf
Kanazawa, 920-1192, Japan
}
\email{y-kwkami@se.kanazawa-u.ac.jp}


\maketitle

\begin{abstract}
We elucidate the geometric background of function-theoretic properties for the Gauss maps of 
several classes of immersed surfaces in three-dimensional space forms, for example, minimal surfaces in Euclidean three-space, 
improper affine spheres in the affine three-space, and constant mean curvature one surfaces and flat surfaces in hyperbolic three-space. 
To achieve this purpose, we prove an optimal curvature bound for a specified conformal metric on an open Riemann surface and 
give some applications. We also provide unicity theorems for the Gauss maps of these classes of surfaces. 
\end{abstract}

\section{Introduction} 
One of the central issues in minimal surface theory is to understand the global behavior of the Gauss map. 
In the latter half of the twentieth century, Osserman \cite{Os1959, Os1964, Os1986} initiated a systematic study of the Gauss map and, 
in particular, proved that the image of the Gauss map of a nonflat complete minimal surface in Euclidean 3-space ${\R}^{3}$ must be dense 
in the unit 2-sphere ${\Si}^{2}$. 
Xavier \cite{Xa1981} then showed that the Gauss map can omit at most a finite number of values in ${\Si}^{2}$, and 
Fujimoto \cite{Fu1988} proved that the precise maximum for the number of omitted values possible is 4. Fujimoto also 
gave a curvature bound for a minimal surface when all of the multiple values of the Gauss map are totally ramified (\cite{Fu1992, Fu1993}). 
Here a value $\alpha$ of a map or function $g$ is said to be {\it totally ramified} if the equation $g=\alpha$ has no simple roots. 
Moreover, Fujimoto obtained a unicity theorem for the Gauss maps of nonflat complete minimal surfaces, 
which is analogous to the Nevanlinna unicity theorem (\cite{Ne1926}) 
for meromorphic functions on the complex plane $\C$ (\cite{Fu1993-2}). 

There exist several classes of immersed surfaces whose Gauss maps have these function-theoretical properties. 
For instance, Yu \cite{Yu1997} showed that the hyperbolic Gauss map of a nonflat complete constant mean curvature one surface 
in hyperbolic 3-space ${\H}^{3}$ can omit at most 4 values. The author and Nakajo \cite{KN2012} obtained that the maximal number 
of omitted values of the Lagrangian Gauss map of a weakly complete improper affine front in the affine 3-space ${\R}^{3}$ is 3, unless 
it is an elliptic paraboloid. As an application of this result, a simple proof of the parametric affine Bernstein theorem for an improper 
affine sphere in ${\R}^{3}$ was provided. Moreover the author \cite{Ka2013-2} gave similar results for flat fronts in ${\H}^{3}$. 

In \cite{Ka2013}, we revealed a geometric meaning for the maximal number of omitted values of their Gauss maps. To be precise, 
we gave a curvature bound for the conformal metric $ds^{2}=(1+|g|^{2})^{m}|\omega|^{2}$ on an open Riemann surface $\Sigma$, 
where $\omega$ is a holomorphic 1-form and $g$ is a meromorphic function on $\Sigma$ (\cite[Theorem 2.1]{Ka2013}) and, as a corollary of 
the theorem, proved that the precise maximal number of omitted values of the nonconstant meromorphic function $g$ on $\Sigma$ with the 
complete conformal metric $ds^{2}$ is $m+2$ (\cite[Corollary 2.2, Proposition 2.4]{Ka2013}). Since the induced metric from ${\R}^{3}$ of a 
complete minimal surface is $ds^{2}=(1+|g|^{2})^{2}|\omega|^{2}$ (i.e. $m=2$), the maximum number of omitted values of the Gauss map $g$ of 
a nonflat complete minimal surface in ${\R}^{3}$ is $4\,(=2+2)$. On the other hand, for the Lagrangian Gauss map $\nu$ of a weakly complete 
improper affine front, because $\nu$ is meromorphic, $dG$ is holomorphic and the complete metric is $d{\tau}^{2}=(1+|\nu|^{2})|dG|^{2}$ (i.e. $m=1$), 
the maximal number of omitted values of the Lagrangian Gauss map of a weakly complete improper affine front in ${\R}^{3}$ is 
$3\,(=1+2)$, unless it is an elliptic paraboloid. 

The goal of this paper is to elucidate the geometric background of function-theoretic properties for the Gauss maps. 
The paper is organized as follows: In Section 2, we first give a curvature bound for the conformal metric $ds^{2}=(1+|g|^{2})^{m}|\omega|^{2}$ 
on an open Riemann surface $\Sigma$ when all of the multiple values of the meromorphic function $g$ are 
totally ramified (Theorem \ref{thm-ramification}). 
This is a generalization of Theorem 2.1 in \cite{Ka2013}, and the proof is given in Section 3.1. 
As a corollary of this theorem, we give a ramification theorem for the meromorphic function $g$ on $\Sigma$ 
with the complete conformal metric $ds^{2}$ (Corollary \ref{cor-ramification}). We remark that this corresponds to the defect relation 
in Nevanlinna theory (see \cite{Ko2003}, \cite{NO1990}, \cite{NW2013}, \cite{Ru2001} for the details). 
Next, we provide two applications of the result. The first one is to show that the precise maximal number of 
omitted values of the nonconstant meromorphic function $g$ on $\Sigma$ with complete conformal 
metric $ds^{2}$ is $m+2$ (Corollary \ref{cor-exceptional}). The second one is to prove an analogue of 
a special case of the Ahlfors islands theorem \cite[Theorem B.2]{Be2000} for $g$ on $\Sigma$ 
with the complete conformal metric $ds^{2}$ (Corollaries \ref{cor-covering} and \ref{cor-Ahlfors-island}). 
The Ahlfors islands theorem has found various applications in complex dynamics; see \cite{Be2000} for an exposition. 
We also give a unicity theorem for the nonconstant meromorphic function $g$ on an open Riemann surface $\Sigma$ with the complete conformal metric $ds^{2}$ 
(Theorem \ref{thm-unicity}). This theorem is optimal in that for every even number $m$, there exist examples (Example \ref{exa-unicity}). 
The proof is given in Section 3.2. When $m=0$, all results coincide with the results 
for meromorphic functions on $\C$ (Remarks \ref{rem-geometry}, \ref{rem-Ahlfors-island} and \ref{rem-unicity-thm}). 
In Section 4, as applications of the main results, we show some function-theoretic properties for the Gauss maps of the following classes 
of surfaces: minimal surfaces in ${\R}^{3}$ (Section 4.1), constant mean curvature one surfaces in ${\H}^{3}$ (Section 4.2), 
maxfaces in ${\R}^{3}_{1}$ (Section 4.3), improper affine fronts in ${\R}^{3}$ (Section 4.4) and flat fronts in ${\H}^{3}$ (Section 4.5). 
In particular, we give their geometric background. 

Finally, the author would like to thank Professors Junjiro Noguchi, Wayne Rossman, Masaaki Umehara and Kotaro Yamada for their useful advice. 
The author also would like to express his thanks to Professors Ryoichi Kobayashi, Masatoshi Kokubu, Miyuki Koiso and Reiko Miyaoka 
for their constant encouragement. 


\section{Main results}
\subsection{Curvature bound and its corollaries} 
We first give the following curvature bound for the conformal metric $ds^{2}=(1+|g|^{2})^{m}|\omega|^{2}$ on an 
open Riemann surface $\Sigma$. This is more precise than Theorem 2.1 in \cite{Ka2013}. 
\begin{theorem}\label{thm-ramification}
Let $\Sigma$ be an open Riemann surface with the conformal metric 
\begin{equation}\label{equ-metric}
ds^{2}=(1+|g|^{2})^{m}|\omega|^{2}, 
\end{equation}
where $\omega$ is a holomorphic $1$-form, $g$ is a meromorphic function on $\Sigma$, and $m\in \N$. 
Let $q\in \N$, ${\alpha}_{1}, \ldots, {\alpha}_{q}\in \C\cup \{\infty \}$ be distinct and ${\nu}_{1}, \ldots, {\nu}_{q}\in \N\cup \{\infty \}$. 
Suppose that
\begin{equation}\label{equ-ramification}
\gamma= \displaystyle \sum_{j=1}^{q} \biggl{(}1-\dfrac{1}{{\nu}_{j}} \biggr{)}> m+2. 
\end{equation}
If $g$ satisfies the property that all ${\alpha}_{j}$-points of $g$ have multiplicity at least ${\nu}_{j}$, 
then there exists a positive constant $C$, depending on $m$, $\gamma$ and ${\alpha}_{1}, \ldots, {\alpha}_{q}$, but not the surface, 
such that for all $p\in \Sigma$ we have
\begin{equation}\label{equ-curvature}
|K_{ds^{2}}(p)|^{1/2}\leq \dfrac{C}{d(p)}, 
\end{equation}
where $K_{ds^{2}}(p)$ is the Gaussian curvature of the metric $ds^{2}$ at $p$ and 
$d(p)$ is the geodesic distance from $p$ to the boundary of $\Sigma$, that is, 
the infimum of the lengths of the divergent curves in $\Sigma$ emanating from $p$. 
\end{theorem}

As a corollary of Theorem \ref{thm-ramification}, we give the following ramification theorem for the meromorphic function $g$ on $\Sigma$ 
with the complete conformal metric $ds^{2}=(1+|g|^{2})^{m}|\omega|^{2}$. 

\begin{corollary}\label{cor-ramification} 
Let $\Sigma$ be an open Riemann surface with the conformal metric given by (\ref{equ-metric}). 
Let $q\in \N$, ${\alpha}_{1}, \ldots, {\alpha}_{q}\in \C\cup \{\infty \}$ be distinct and ${\nu}_{1}, \ldots, {\nu}_{q}\in \N\cup \{\infty \}$. 
Suppose that the metric $ds^{2}$ is complete and that the inequality (\ref{equ-ramification}) holds. 
If $g$ satisfies the property that all ${\alpha}_{j}$-points of $g$ have multiplicity at least ${\nu}_{j}$, 
then $g$ must be constant. 
\end{corollary}

\begin{proof}
Since $ds^{2}$ is complete, we may set $d(p)=\infty$ for all $p\in \Sigma$. By virtue of Theorem \ref{thm-ramification}, 
$K_{ds^{2}}\equiv 0$ on $\Sigma$. On the other hand, the Gaussian curvature of the metric $ds^{2}$ is given by 
\begin{equation}\label{equ-Gaussian}
K_{ds^{2}}=-\dfrac{2m|g'_{z}|^{2}}{(1+|g|^{2})^{m+2}|\hat{\omega}_{z}|^{2}}, 
\end{equation}
where $\omega= \hat{\omega}_{z}dz$ and $g'_{z}=dg/dz$. Hence $K_{ds^{2}}\equiv 0$ if and only if $g$ is constant. 
\end{proof}

\begin{remark}\label{rem-geometry}
The geometric meaning of the ``2'' in ``$m+2$'' is the Euler number of the Riemann sphere. 
Indeed, if $m=0$ then the metric $ds^{2}=(1+|g|^{2})^{0}|\omega|^{2}=|\omega|^{2}$ is flat and complete on $\Sigma$. 
We thus may assume that $g$ is a meromorphic function on $\C$ because $g$ is replaced by $g\circ \pi$, 
where $\pi\colon \C\to \Sigma$ is a holomorphic universal covering map. On the other hand, Ahlfors \cite{Ah1935} and Chern \cite{Ch1960} 
showed that the least upper bound for the defect relation for meromorphic functions on $\C$ coincides with the Euler number of the Riemann 
sphere. Hence we get the conclusion. 
\end{remark}

We next give two applications of Corollary \ref{cor-ramification}. 
The first one is to provide the precise maximal number of omitted values of 
the meromorphic function $g$ on $\Sigma$ with the complete conformal metric $ds^{2}=(1+|g|^{2})^{m}|\omega|^{2}$. 

\begin{corollary}[{\cite[Corollary 2.2]{Ka2013}}]\label{cor-exceptional}
Let $\Sigma$ be an open Riemann surface with the conformal metric given by (\ref{equ-metric}). 
If the metric $ds^{2}$ is complete and the meromorphic function $g$ is nonconstant, then $g$ can omit at most $m+2$ distinct values. 
\end{corollary}

\begin{proof}
By way of contradiction, assume that $g$ omits $m+3$ distinct values. In Corollary \ref{cor-ramification}, 
if $g$ does not take a value ${\alpha}_{j}$ $(j=1, \ldots, q)$, 
we may set ${\nu}_{j}=\infty$ in (\ref{equ-ramification}). Thus we can consider the case where $\gamma \geq m+3\, (>m+2)$. 
By virtue of Corollary \ref{cor-ramification}, the function $g$ is constant. This contradicts the assumption that $g$ is nonconstant. 
\end{proof}
The number ``$m+2$'' is sharp because there exist examples in \cite[Proposition 2.4]{Ka2013}. 

The second one is to show an analogue of the Ahlfors islands theorem \cite[Theorem B.2]{Be2000} for 
the meromorphic function $g$ on $\Sigma$ with the complete conformal metric $ds^{2}=(1+|g|^{2})^{m}|\omega|^{2}$. 
We first recall the notion of chordal distance between two distinct values in the Riemann sphere $\C\cup \{\infty \}$. 
For two distinct values $\alpha$, $\beta\in \C\cup \{\infty\}$, we set 
$$
|\alpha, \beta|:= \dfrac{|\alpha -\beta|}{\sqrt{1+|\alpha|^{2}}\sqrt{1+|\beta|^{2}}}
$$
if $\alpha \not= \infty$ and $\beta \not= \infty$, and $|\alpha, \infty|=|\infty, \alpha| := 1/\sqrt{1+|\alpha|^{2}}$. 
We remark that, if we take $v_{1}$, $v_{2}\in {\Si}^{2}$ with $\alpha =\varpi (v_{1})$ and $\beta = \varpi (v_{2})$, we have that 
$|\alpha, \beta|$ is a half of the chordal distance between $v_{1}$ and $v_{2}$, where $\varpi$ denotes the stereographic projection of 
${\Si}^{2}$ onto $\C\cup \{\infty \}$. We next explain the definition of an island of a meromorphic function on a Riemann surface. 

\begin{definition}\label{def-island}
Let $\Sigma$ be a Riemann surface and $g\colon \Sigma\to \C\cup\{\infty\}$ a meromorphic function. 
Let $V\subset \C\cup\{\infty\}$ be a Jordan domain. A simply-connected component $U$ of $g^{-1}(V)$ with $\overline{U}\subset \Sigma$ is 
called an {\it island} of $g$ over $V$. Note that $g|_{U}\colon U\to V$ is a proper map. The degree of this map is called the {\it multiplicity} of 
the island $U$. An island of multiplicity one is called a {\it simple island}. 
\end{definition} 

When all islands of the meromorphic function $g$ with the complete conformal metric $ds^{2}=(1+|g|^{2})^{m}|\omega|^{2}$ are small disks, 
we get the following result by applying Corollary \ref{cor-ramification}. 

\begin{corollary}\label{cor-covering}
Let $\Sigma$ be an open Riemann surface with the conformal metric given by (\ref{equ-metric}). 
Let $q\in \N$, ${\alpha}_{1}, \ldots, {\alpha}_{q}\in \C\cup \{\infty \}$ be distinct, 
$D_{j}({\alpha}_{j}, \varepsilon):=\{z\in \C\cup \{\infty \} \,;\, |z, {\alpha}_{j}|< \varepsilon \}$ $(1\leq j\leq q)$ be pairwise disjoint and 
${\nu}_{1}, \ldots, {\nu}_{q}\in \N$. Suppose that the metric $ds^{2}$ is complete and that the inequality (\ref{equ-ramification}) holds. 
Then there exists $\varepsilon > 0$ such that, if $g$ has no island of multiplicity less than ${\nu}_{j}$ over $D_{j}({\alpha}_{j}, \varepsilon)$ 
for all $j\in \{1, \ldots , q\}$, then $g$ must be constant. 
\end{corollary}

\begin{proof}
If such an $\varepsilon$ does not exist, for any $\varepsilon$ we can find a nonconstant meromorphic function $g$ which has no island of 
multiplicity less than ${\nu}_{j}$ over $D_{j}({\alpha}_{j}, \varepsilon)$. However this implies that all ${\alpha}_{j}$-points of $g$ have 
multiplicity at least ${\nu}_{j}$, contradicting Corollary \ref{cor-ramification}. 
\end{proof}

The important case of Corollary \ref{cor-covering} is the case where $q=2m+5$ and ${\nu}_{j}=2$ for each $j$ $(j=1, \ldots, q)$. 
This corresponds to the so-called five islands theorem in the Ahlfors theory of covering surfaces (\cite{Ah1935}, \cite[Chapter XIII]{Ne1970}). 

\begin{corollary}\label{cor-Ahlfors-island}
Let $\Sigma$ be an open Riemann surface with the complete conformal metric given by (\ref{equ-metric}). Let ${\alpha}_{1}, \ldots, 
{\alpha}_{2m+5} \in \C\cup \{\infty \}$ be distinct and 
$D_{j}({\alpha}_{j}, \varepsilon):=\{z\in \C\cup \{\infty \} \,;\, |z, {\alpha}_{j}|< \varepsilon \}$ $(1\leq j\leq 2m+5)$. 
Then there exists $\varepsilon > 0$ such that, if $g$ has no simple island of over any of the small disks $D_{j}({\alpha}_{j}, \varepsilon)$, 
then $g$ must be constant. 
\end{corollary}

\begin{remark}\label{rem-Ahlfors-island}
Theorem \ref{cor-Ahlfors-island} is valid for the case where $m=0$. In fact, by the same argument in Remark \ref{rem-geometry}, 
we can easily show that the theorem corresponds to a special case of the Ahlfors five islands theorem. 
\end{remark}

\subsection{Unicity theorem} 
We give another type of function-theoretic property of the meromorphic function $g$ on $\Sigma$ with the complete 
conformal metric $ds^{2}=(1+|g|^{2})^{m}|\omega|^{2}$. In \cite{Ne1926}, Nevanlinna showed that two nonconstant meromorphic functions 
on $\C$ coincides with each other if they have the same inverse images for five distinct values. We get the following analogue to 
this unicity theorem.

\begin{theorem}\label{thm-unicity}
Let $\Sigma$ be an open Riemann surface with the conformal metric 
\begin{equation}\label{equ-metric1}
ds^{2}=(1+|g|^{2})^{m}|\omega|^{2}, 
\end{equation}
and $\widehat{\Sigma}$ be another open Riemann surface with the conformal metric 
\begin{equation}\label{equ-metric2}
d{\hat{s}}^{2}=(1+|\hat{g}|^{2})^{m}|\hat{\omega}|^{2}, 
\end{equation}
where $\omega$ and $\hat{\omega}$ are holomorphic $1$-forms, $g$ and $\hat{g}$ are nonconstant meromorphic functions on $\Sigma$ and $\widehat{\Sigma}$ respectively, 
and $m\in \N$. We assume that there exists a conformal diffeomorphism $\Psi\colon \Sigma \to \widehat{\Sigma}$. Suppose that there exist $q$ distinct points 
${\alpha}_{1}, \ldots, {\alpha}_{q}\in \C\cup \{\infty \}$ such that $g^{-1}({\alpha}_{j})=(\hat{g}\circ \Psi)^{-1}({\alpha}_{j})$ $(1\leq j\leq q)$. 
If $q \geq m+5 \,(=(m+4)+1)$ and either $ds^{2}$ or $d{\hat{s}}^{2}$ is complete, then $g\equiv \hat{g}\circ \Psi$. 
\end{theorem}

\begin{remark}\label{rem-unicity-thm}
When $m=0$, Theorem \ref{thm-unicity} coincides with the Nevanlinna unicity theorem. 
\end{remark}

The maps $g$ and $\hat{g}\circ \Psi$ are said to share the value $\alpha$ (ignoring multiplicity) when 
$g^{-1}(\alpha)= (\hat{g}\circ \Psi)^{-1}(\alpha)$. 
Theorem \ref{thm-unicity} is optimal for an arbitrary even number $m\, (\geq 2)$ because there exist the following examples. 
\begin{example}\label{exa-unicity}
For an arbitrary even number $m\, (\geq 2)$, we take $m/2$ distinct points ${\alpha}_{1}, \ldots, {\alpha}_{m/2}$ in $\C\backslash \{0, \pm 1 \}$. 
Let $\Sigma$ be either the complex plane punctured at $m+1$ distinct points $0$, ${\alpha}_{1}, \ldots, {\alpha}_{m/2}$, 
$1/{\alpha}_{1}, \ldots, 1/{\alpha}_{m/2}$ or the universal covering of that punctured plane. We set 
$$
\omega = \dfrac{dz}{z\prod_{i=1}^{m/2} (z-{\alpha}_{i})({\alpha}_{i}z-1)}, \quad g(z)=z\,,
$$
and 
$$
\hat{\omega}\, (=\omega) = \dfrac{dz}{z\prod_{i=1}^{m/2} (z-{\alpha}_{i})({\alpha}_{i}z-1)}, \quad \hat{g}(z)=\dfrac{1}{z}. 
$$
We can easily show that the identity map $\Psi\colon \Sigma \to \Sigma$ is a conformal diffeomorphism and the metrics $ds^{2}=(1+|g|^{2})^{m}|\omega|^{2}$ 
and $d\hat{s}^{2}=(1+|\hat{g}|^{2})^{m}|\hat{\omega}|^{2}$ are complete. Then the maps $g$ and $\hat{g}$ share the $m+4$ distinct values 
$0,\, \infty,\, 1,\, -1,\, {\alpha}_{1},\,\ldots, {\alpha}_{m/2},\, 1/{\alpha}_{1},\, \ldots,\, 1/{\alpha}_{m/2}$ and $g\not\equiv \hat{g}\circ \Psi$. 
These show that the number $m+5$ in Theorem \ref{thm-unicity} cannot be replaced by $m+4$. 
\end{example}
\section{Proof of main theorems}
\subsection{Proof of Theorem \ref{thm-ramification}} 
Before proceeding to the proof of Theorem \ref{thm-ramification}, we recall two lemmas. 

\begin{lemma}[{\cite[Corollary 1.4.15]{Fu1993}}]\label{main-lem1}
Let $g$ be a nonconstant meromorphic function on ${\triangle}_{R}=\{z\in \C; |z|< R\}$ $(0<R\leq \infty)$. 
Let $q\in \N$, ${\alpha}_{1}, \ldots, {\alpha}_{q}\in \C\cup\{\infty\}$ be distinct 
and ${\nu}_{1}, \ldots, {\nu}_{q}\in \N\cup \{\infty \}$. Suppose that 
$$
\gamma =\displaystyle \sum_{j=1}^{q} \biggl(1-\dfrac{1}{{\nu}_{j}} \biggr)> 2. 
$$
If $g$ satisfies the property that all ${\alpha}_{j}$-points of $g$ have multiplicity at least ${\nu}_{j}$, then, for 
arbitrary constants $\eta\geq 0$ and $\delta >0$ with $\gamma -2>\gamma(\eta +\delta)$, then there exists a positive constant $C'$, 
depending only on $\gamma$, $\eta$, $\delta$, and $L:={\min}_{i<j}|{\alpha}_{i}, {\alpha}_{j}|$, such that 
\begin{equation}\label{eq-lem-estimate}
\dfrac{|g'|}{1+|g|^{2}}\dfrac{1}{({\prod}_{j=1}^{q}|g, {\alpha}_{j}|^{1-1/{\nu}_{j}})^{1-\eta-\delta}}\leq C'\dfrac{R}{R^{2}-|z|^{2}}. 
\end{equation}
\end{lemma}

\begin{lemma}[{\cite[Lemma 1.6.7]{Fu1993}}]\label{main-lem2}
Let $d{\sigma}^{2}$ be a conformal flat metric on an open Riemann surface $\Sigma$. 
Then, for each point $p\in \Sigma$, there exists a local diffeomorphism $\Phi$ of a 
disk ${\Delta}_{R}=\{z\in \C\, ;\, |z|<{R}\}$ $(0<{R}\leq +\infty)$ onto an open neighborhood 
of $p$ with $\Phi (0)=p$ such that $\Phi$ is a local isometry, that is, the pull-back ${\Phi}^{\ast}(d{\sigma}^{2})$ 
is equal to the standard Euclidean metric $ds_{Euc}^{2}$ on  ${\Delta}_{R}$ and, for a point $a_{0}$ with $|a_{0}|=1$, 
the $\Phi$-image ${\Gamma}_{a_{0}}$ of the curve $L_{a_{0}}=\{w:=a_{0}s\, ;\, 0<s<R \}$ is divergent in $\Sigma$. 
\end{lemma}

\begin{proof}[{\it Proof of Theorem \ref{thm-ramification}}] 
For the proof of Theorem \ref{thm-ramification}, we may assume the following: 
\begin{enumerate}
\item[(A)] For any proper subset $I$ in $\{1, 2, \ldots, q\}$, 
$$
\displaystyle \sum_{j\in I}\biggl{(}1-\dfrac{1}{{\nu}_{j}} \biggr{)}\leq m+2.  
$$
\item[(B)] There exists no set of positive integers $({\nu}^{\ast}_{1}, \ldots, {\nu}^{\ast}_{q})$ distinct with 
$({\nu}_{1}, \ldots, {\nu}_{q})$ satisfying the conditions 
\begin{equation}\label{equ-proof-assume}
{\nu}^{\ast}_{j}\leq {\nu}_{j} \; (1\leq j\leq q), \quad \displaystyle \sum_{j=1}^{q} \biggl{(}1-\dfrac{1}{{\nu}^{\ast}_{j}} \biggr{)}> m+2.
\end{equation}
\end{enumerate}

If there exists some proper subset $I$ in $\{1, 2, \ldots, q\}$ such that 
$$
\displaystyle \sum_{j\in I}\biggl{(}1-\dfrac{1}{{\nu}_{j}} \biggr{)}> m+2, 
$$
then the assumption in Theorem \ref{thm-ramification} for $\{{\alpha}_{j}\,;\, 1\leq j\leq q\}$ can be replace by 
the assumption for $\{{\alpha}_{j}\,;\, j\in I \}$. Moreover if there exist some $({\nu}^{\ast}_{1}, \ldots, {\nu}^{\ast}_{q})$ satisfying 
the conditions (\ref{equ-proof-assume}), then we may prove Theorem \ref{thm-ramification} after replacing each integer ${\nu}_{j}$ by ${\nu}^{\ast}_{j}$.  

\begin{lemma}\label{lem-finite}
There exist only finite many sets of integers ${\nu}_{1}, \ldots, {\nu}_{q}$ with ${\nu}_{j}\geq 2$ which satisfy the conditions (A) and (B). 
\end{lemma}
\begin{proof}
We take positive integers ${\nu}_{1}, \ldots, {\nu}_{q}$ satisfying the conditions (A) and (B). 
We may assume that ${\nu}_{1}\leq \ldots \leq {\nu}_{q}$. 
Then, for the number 
$$
\displaystyle \gamma =\sum_{j=1}^{q}\biggl{(}1-\dfrac{1}{{\nu}_{j}} \biggr{)}, 
$$
we shall show that
\begin{equation}\label{equ-proof-331}
\gamma - (m+2) \leq \dfrac{1}{{\nu}_{q}({\nu}_{q}-1)}. 
\end{equation}
In fact, we suppose that $\gamma - (m+2)> 1/{\nu}_{q}({\nu}_{q}-1)$. If ${\nu}_{q}=2$, then 
$$
\gamma > (m+1)+\dfrac{1}{2}
$$
and
$$ 
\displaystyle \sum_{j=1}^{q-1} \biggl{(}1-\dfrac{1}{{\nu}_{j}} \biggr{)}> m+2, 
$$
which contradicts the assumption (A). Thus, ${\nu}_{q}\geq 3$. Here, if we set ${\nu}^{\ast}_{j}:={\nu}_{j}$ $(1\leq j\leq q-1)$ and 
${\nu}^{\ast}_{q}:={\nu}_{q}-1$, then 
$$
\displaystyle \sum_{j=1}^{q}\biggl{(}1-\dfrac{1}{{\nu}^{\ast}_{j}} \biggr{)}= \sum_{j=1}^{q}\biggl{(}1-\dfrac{1}{{\nu}_{j}} \biggr{)}-\dfrac{1}{{\nu}_{q}({\nu}_{q}-1)}> m+2. 
$$
This contradicts the assumption (B). We have thus proved the inequality (\ref{equ-proof-331}). 

By virtue of (\ref{equ-proof-331}), 
$$
\gamma \leq m+2+\dfrac{1}{{\nu}_{q}({\nu}_{q}-1)}< m+2+\dfrac{1}{2}= m+\dfrac{5}{2}. 
$$
On the other hand, since ${\nu}_{j}\geq 2$ for all $j$, we have 
$$
\gamma =\displaystyle \sum_{j} \biggl{(}1-\dfrac{1}{{\nu}_{j}} \biggr{)}\geq q\biggl{(}1-\dfrac{1}{{\nu}_{1}} \biggr{)}\geq \dfrac{q}{2}, 
$$
where $q\geq 5$. Hence we obtain that ${\nu}_{1}<2q/(2q-2m-5)$ and $q< 2m+5$. 

Now we consider the numbers ${\nu}_{1}, \ldots, {\nu}_{q}$ satisfying the conditions (A) and (B). We set 
$$
{\gamma}_{0}:= \displaystyle \sum_{j=1}^{k} \biggl{(}1-\dfrac{1}{{\nu}_{j}} \biggr{)}. 
$$
Since 
$$
m+2 < \gamma = {\gamma}_{0}+ \displaystyle \sum_{j=k+1}^{q}\biggl{(}1-\dfrac{1}{{\nu}_{j}} \biggr{)}, 
$$
we have 
$$
{\gamma}_{0}+q-k-(m+2)=\gamma -(m+2)+\displaystyle \sum_{j=k+1}^{q}\dfrac{1}{{\nu}_{j}}> 0. 
$$
Then we take a number $N$ with ${\gamma}_{0}+q-k-(m+2)>{\eta}_{0}:=1/N(N-1)$. If ${\nu}_{q}\leq N$, then we have ${\nu}_{k+1}\leq N$. 
Otherwise, by the inequality (\ref{equ-proof-331}) and ${\nu}_{k}\leq {\nu}_{j}$ for $j=k+1, \ldots, q$, we have 
\begin{eqnarray}
0<{\gamma}_{0}+q-k-{\eta}_{0}-(m+2) &\leq& {\gamma}_{0}+q-k-(m+2)-\dfrac{1}{{\nu}_{q}({\nu}_{q}-1)}  \nonumber \\
                                    &\leq& \displaystyle \sum_{j=k+1}^{q}\dfrac{1}{{\nu}_{j}}\leq \dfrac{q-k}{{\nu}_{k+1}}. \nonumber
\end{eqnarray}
This gives 
$$
{\nu}_{k+1}\leq \dfrac{q-k}{{\gamma}_{0}-(m+2)+q-k-{\eta}_{0}}. 
$$
We thereby get that 
$$
{\nu}_{k+1}\leq \max\biggl{\{}N,\, \dfrac{q-k}{{\gamma}_{0}-(m+2)+q-k-{\eta}_{0}}\biggr{\}}. 
$$
Since the boundedness of ${\nu}_{1}$ has been already shown, by induction on $k$ $(=1, \ldots, q)$, we have completed the proof of the lemma. 
\end{proof} 

By Lemma \ref{lem-finite}, if we take the maximum $C_{0}$ in constants which are chosen for the finitely many possible cases of ${\nu}_{j}'s$ 
satisfying the conditions (A) and (B), then $C_{0}$ satisfies the desired inequality (\ref{equ-curvature}). Hence, for the proof of Theorem \ref{thm-ramification}, 
we shall show the existence of a constant satisfying (\ref{equ-curvature}) which may depend on the given data ${\nu}_{1}, \ldots, {\nu}_{q}$.  

We may assume that $m\not= 0$ and ${\alpha}_{q}=\infty$ after a suitable M\"obius transformation. We choose some $\delta$ such that $\gamma -(m+2)> 2\gamma\delta >0$ 
and set $m\not=0$ and 
\begin{equation}\label{equ-proof-211}
\eta :=\dfrac{\gamma -(m+2)-2\gamma\delta}{\gamma}, \quad \lambda :=\dfrac{m}{m+\gamma\delta}. 
\end{equation}
Then if we choose a sufficiently small positive number $\delta$ depending on $\gamma$ and $m$, for the constant ${\varepsilon}_{0}:=
(\gamma -(m+2))/2m\gamma$ we have 
\begin{equation}\label{equ-proof-212}
0<\lambda <1, \quad \dfrac{{\varepsilon}_{0}\lambda}{1-\lambda}\biggl{(}=\dfrac{\gamma -(m+2)}{2\delta{\gamma}^{2}} \biggr{)}>1. 
\end{equation}

Now we define a new metric 
\begin{equation}\label{equ-proof-213}
\displaystyle d{\sigma}^{2}=|\hat{\omega}_{z}|^{2/(1-\lambda)}
\biggl{(}\dfrac{1}{|g'_{z}|}\prod_{j=1}^{q-1}\biggl{(}\dfrac{|g-{\alpha}_{j}|}{\sqrt{1+|{\alpha}_{j}|^{2}}} \biggr{)}^{{\mu}_{j}(1-\eta -\delta)} 
\biggr{)}^{2\lambda /(1-\lambda)} |dz|^{2}
\end{equation}
on the set ${\Sigma}'=\{p\in \Sigma \,;\, g'_{z}\not= 0\; \text{and}\; g(z)\not= {\alpha}_{j}\; \text{for all}\; j\}$, 
where $\omega =\hat{\omega}_{z}dz$, $g'_{z}=dg/dz$ and ${\mu}_{j}=1-(1/{\nu}_{j})$. Take a point $p\in {\Sigma}'$. Since the metric $d{\sigma}^{2}$ is flat on ${\Sigma}'$, 
by Lemma \ref{main-lem2}, there exists a local isometry $\Phi$ satisfying $\Phi (0)=p$ from a disk 
${\triangle}_{R}=\{ z\in \C\,;\, |z|<R \}$ $(0< R\leq +\infty)$ with the standard metric $ds^{2}_{Euc}$ onto an open neighborhood of $p$ 
in ${\Sigma}'$ with the metric $d{\sigma}^{2}$ such that for a point $a_{0}$ with $|a_{0}|=1$, the $\Phi$-image ${\Gamma}_{a_{0}}$ of the 
curve $L_{a_{0}}=\{w:=a_{0}s\,;\, 0<s<R \}$ is divergent in ${\Sigma}'$. For brevity, we denote the function $g\circ \Phi$ on ${\triangle}_{R}$ by 
$g$ in the following. By Lemma \ref{main-lem1}, we get that 
\begin{equation}\label{equ-proof-214}
R\leq C'\dfrac{1+|g(0)|^{2}}{|g'_{z}(0)|} \displaystyle \prod_{j=1}^{q}|g(0), {\alpha}_{j}|^{{\mu}_{j}(1-\eta -\delta)} < +\infty. 
\end{equation} 
Hence 
$$
L_{d\sigma} ({\Gamma}_{a_{0}})=\int_{{\Gamma}_{a_{0}}} d\sigma =R < +\infty , 
$$
where $L_{d\sigma} ({\Gamma}_{a_{0}})$ denotes the length of ${\Gamma}_{a_{0}}$ with respect to the metric $d{\sigma}^{2}$. 

Now we prove that ${\Gamma}_{a_{0}}$ is divergent in $\Sigma$. Indeed, if not, then ${\Gamma}_{a_{0}}$ must tend to a point 
$p_{0}\in \Sigma\backslash {\Sigma}'$ where $g'_{z}(p_{0})= 0$ or $g(p_{0})={\alpha}_{j}$ for some $j$ because ${\Gamma}_{a_{0}}$ is 
divergent in ${\Sigma}'$ and $L_{d\sigma} ({\Gamma}_{a_{0}})< +\infty$. Taking a local complex coordinate $\zeta$ in a neighborhood of 
$p_{0}$ with $\zeta (p_{0})=0$, we can write the metric $d{\sigma}^{2}$ as 
$$
d{\sigma}^{2}= |\zeta|^{2k\lambda /(1-\lambda)} w|d\zeta|^{2}, 
$$
with some positive smooth function $w$ and some real number $k$. If $g-{\alpha}_{j}$ has a zero of order $l\,(\geq {\nu}_{j}\geq 2)$ at $p_{0}$ 
for some $1\leq j\leq q-1$, then $g'_{z}$ has a zero of order $l-1$ at $p_{0}$ and $\hat{\omega}_{z}(z_{0})\not= 0$. Then we obtain that 
\begin{eqnarray}
k &=& -(l-1)+l\biggl{(}1-\dfrac{1}{{\nu}_{j}} \biggr{)}(1-\eta -\delta) \nonumber \\
  &=& \biggl{(}1-\dfrac{l}{{\nu}_{j}} \biggr{)}-\dfrac{l}{{\nu}_{j}}({\nu}_{j}-1)(\eta +\delta) \nonumber \\
  &=& -(\eta +\delta)\leq -{\varepsilon}_{0}. \nonumber
\end{eqnarray}
For the case where $g$ has a pole of order $l\,(\geq {\nu}_{q}\geq 2)$, $g'_{z}$ has a pole of order $l+1$, $\hat{\omega}_{z}$ has a zero of 
order $ml$ at $p_{0}$ and each component $g-{\alpha}_{j}$ in the right side of (\ref{equ-proof-213}) has a pole of order $l$ at $p_{0}$. 
Using the identity ${\mu}_{1}+\cdots +{\mu}_{q-1}=\gamma -{\mu}_{q}$ and (\ref{equ-proof-211}), we get that 
\begin{eqnarray}
 k &=& \dfrac{ml}{\lambda} +(l+1)-l(\gamma -{\mu}_{q})(1-\eta -\delta) \nonumber \\
   &=& l{\mu}_{q}(1-\eta -\delta) -(l-1)\leq -{\varepsilon}_{0}. \nonumber
\end{eqnarray}
Moreover, for the case where $g'_{z}(p_{0})= 0$ and $g(p_{0})\not= {\alpha}_{j}$ for all $j$, we see that $k\leq -1$. 
In any case, $k\lambda /(1-\lambda)\leq -1$ by (\ref{equ-proof-212}) and there exists a positive constant $\widetilde{C}$ such that 
$$
d{\sigma}\geq \widetilde{C}\dfrac{|d\zeta|}{|\zeta|}
$$
in a neighborhood of $p_{0}$. Thus we obtain that 
$$
R=\int_{{\Gamma}_{a_{0}}} d\sigma \geq \widetilde{C}\int_{{\Gamma}_{a_{0}}} \dfrac{|d\zeta|}{|\zeta|}= +\infty, 
$$
which contradicts (\ref{equ-proof-214}). 

Since ${\Phi}^{\ast}d{\sigma}^{2}=|dz|^{2}$, we get by (\ref{equ-proof-213}) that 
\begin{equation}\label{equ-proof-215}
|\hat{\omega}_{z}|=\displaystyle \biggl{(}|g'_{z}| \prod_{j=1}^{q-1}\biggl{(}\dfrac{\sqrt{1+|{\alpha}_{j}|^{2}}}{|g-{\alpha}_{j}|} \biggr{)}^{{\mu}_{j}(1-\eta -\delta)}  \biggr{)}^{\lambda}. 
\end{equation}
By Lemma \ref{main-lem1}, we obtain that 
\begin{eqnarray}
{\Phi}^{\ast}ds &=& (1+|g|^{2})^{m/2}|\omega| \nonumber \\
                &=& \biggl{(}|g'_{z}|(1+|g|^{2})^{m/2\lambda}\displaystyle \prod_{j=1}^{q-1}\biggl{(}\dfrac{\sqrt{1+|{\alpha}_{j}|^{2}}}{|g-{\alpha}_{j}|} \biggr{)}^{{\mu}_{j}(1-\eta -\delta)} \biggr{)}^{\lambda} |dz|\nonumber \\
                &=& \biggl{(}\dfrac{|g'_{z}|}{1+|g|^{2}}\dfrac{1}{\prod_{j=1}^{q} |g, {\alpha}_{j}|^{\mu_{j}(1-\eta -\delta)}} \biggr{)}^{\lambda} |dz|\nonumber \\
                &\leq &(C')^{\lambda}\biggl{(}\dfrac{R}{R^{2}-|z|^{2}} \biggr{)}^{\lambda}|dz|. \nonumber
\end{eqnarray}
Thus we have 
$$
d(p)\leq \int_{{\Gamma}_{a_{0}}} ds = \int_{L_{a_{0}}}{\Phi}^{\ast}ds \leq (C')^{\lambda} \int_{L_{a_{0}}} \biggl{(}\dfrac{R}{R^{2}-|z|^{2}} \biggr{)}^{\lambda}|dz|\leq (C')^{\lambda}\dfrac{R^{1-\lambda}}{1-\lambda}\,(<+\infty)
$$
because $0<\lambda <1$. Moreover, by (\ref{equ-proof-214}), we get that 
$$
d(p)\leq \dfrac{(C')^{\lambda}}{1-\lambda}\biggl{(}\dfrac{1+|g(0)|^{2}}{|g'_{z}(0)|}\displaystyle \prod_{j=1}^{q} |g(0), {\alpha}_{j}|^{{\mu}_{j}(1-\eta -\delta)} \biggr{)}^{1-\lambda}. 
$$
On the other hand, the Gaussian curvature $K_{ds^{2}}$ of the metric $ds^{2}=(1+|g|^{2})^{m}|\omega|^{2}$ is given by 
$$
K_{ds^{2}}=-\dfrac{2m|g'_{z}|^{2}}{(1+|g|^{2})^{m+2}|\hat{\omega}_{z}|^{2}}. 
$$
Thus, by (\ref{equ-proof-215}), we also get that 
$$
|K_{ds^{2}}|^{1/2}=\sqrt{2m}\,\biggl{(}\dfrac{|g'_{z}|}{1+|g|^{2}} \biggr{)}^{1-\lambda}\biggl{(}\displaystyle \prod_{j=1}^{q}|g, {\alpha}_{j}|^{{\mu}_{j}(1-\eta -\delta)} \biggr{)}^{\lambda}. 
$$
Since $|g, {\alpha}_{j}|\leq 1$ for each $j$, we obtain that 
$$
|K_{ds^{2}}(p)|^{1/2}d(p)\leq \dfrac{\sqrt{2m}C'}{1-\lambda}=:C. 
$$
Hence we get the conclusion. 
\end{proof}

\subsection{Proof of Theorem \ref{thm-unicity}}
We review the following two lemmas used in the proof of Theorem \ref{thm-unicity}. 

\begin{lemma}[{\cite[Proposition 2.1]{Fu1993-2}}]\label{main-lem3}
Let $g$ and $\hat{g}$ be mutually distinct nonconstant meromorphic function on a Riemann surface $\Sigma$. 
Let $q\in \N$ and ${\alpha}_{1}, \ldots, {\alpha}_{q}\in \C \cup \{\infty\}$ be distinct. Suppose that 
$q>4$ and $g^{-1}({\alpha}_{j})={\hat{g}}^{-1}({\alpha}_{j})$ $(1\leq j\leq q)$. For $b_{0}>0$ and $\varepsilon$ with 
$q-4> q\varepsilon >0$, we set that 
$$
\xi :=\displaystyle \biggl{(}\prod_{j=1}^{q} |g, {\alpha}_{j}| \log{\biggl{(}\dfrac{b_{0}}{|g, {\alpha}_{j}|^{2}} \biggr{)}}\biggr{)}^{-1+\varepsilon}, 
\quad \hat{\xi} := \displaystyle \biggl{(}\prod_{j=1}^{q} |\hat{g}, {\alpha}_{j}| \log{\biggl{(}\dfrac{b_{0}}{|\hat{g}, {\alpha}_{j}|^{2}} \biggr{)}}\biggr{)}^{-1+\varepsilon}, 
$$
and define that 
\begin{equation}\label{equ-proof-221}
du^{2}:=\biggl{(}|g, \hat{g}|^{2}\xi\, \hat{\xi}\dfrac{|g'|}{1+|g|^{2}}\dfrac{|\hat{g}'|}{1+|\hat{g}|^{2}}\biggr{)}|dz|^{2}
\end{equation}
outside the set $E:=\bigcup_{j=1}^{q}g^{-1}({\alpha}_{j})$ and $du^{2}=0$ on $E$. Then for a suitably chosen $b_{0}$, 
$du^{2}$ is continuous on $\Sigma$ and has strictly negative curvature on the set $\{du^{2}\not= 0\}$. 
\end{lemma}

\begin{lemma}[{\cite[Corollary 2.4]{Fu1993-2}}]\label{main-lem4}
Let $g$ and $\hat{g}$ be a meromorphic function on ${\triangle}_{R}$ satisfying the same assumption as in Lemma \ref{main-lem3}. 
Then for the metric $du^{2}$ defined by (\ref{equ-proof-221}), there exists a constant $\widehat{C}>0$ such that 
$$
du^{2}\leq \widehat{C}\dfrac{R^{2}}{(R^{2}-|z|^{2})^{2}}|dz|^{2}. 
$$
\end{lemma}

\begin{proof}[{\it Proof of Theorem \ref{thm-unicity}}] 
For brevity, we denote the function $\hat{g}\circ \Psi$ by $\hat{g}$ in the following. 
We assume that there exist $q$ distinct values ${\alpha}_{1}, \ldots, {\alpha}_{q}$ such that $g^{-1}({\alpha}_{j})={\hat{g}}^{-1}({\alpha}_{j})$ 
$(1\leq j\leq q)$, ${\alpha}_{q}=\infty$ after a suitable M\"obius transformation and $m\not= 0$. 
Moreover we assume that $q> m+4$, either $ds^{2}$ or $d\hat{s}^{2}$, say $ds^{2}$, is complete and $g\not\equiv \hat{g}$. 
Then the map $\Psi$ gives a biholomorphic isomorphism between $\Sigma$ and $\widehat{\Sigma}$. Thus, for each local complex coordinate $z$ 
defined on a simply connected open domain $U$, we can find  a nonzero holomorphic function $h_{z}$ such that 
\begin{equation}\label{equ-proof-223}
ds^{2}=|h_{z}|^{2}(1+|g|^{2})^{m/2}(1+|\hat{g}|^{2})^{m/2}|dz|^{2}. 
\end{equation}
We take some $\eta$ with $q-(m+4)> q\eta > 0$ and set 
\begin{equation}\label{equ-proof-224}
\lambda :=\dfrac{m}{q-4-q\eta}\:  (< 1). 
\end{equation}
Now we define a new metric 
\begin{equation}\label{equ-proof-225}
d{\sigma}^{2}= |h_{z}|^{2/(1-\lambda)}\Biggl{(}\dfrac{\prod_{j=1}^{q-1}(|g-{\alpha}_{j}||\hat{g}-{\alpha}_{j}|)^{1-\eta}}{|g-\hat{g}|^{2}|g'_{z}||\hat{g}'_{z}|\prod_{j=1}^{q-1}(1+|{\alpha}_{j}|^{2})^{(1-\eta)/2}} \Biggr{)}^{\lambda /(1-\lambda)}|dz|^{2}, 
\end{equation}
on the set ${\Sigma}'=\Sigma \backslash E'$, where 
$$
E'=\{z\in \Sigma \,;\, g'_{z}(z)=0,\, \hat{g}'_{z}(z)=0,\;  \text{or}\; g(z)(=\hat{g}(z))={\alpha}_{j}\; \text{for some $j$} \}. 
$$
On the other hand, setting $\varepsilon := \eta /2$, we can define another pseudo-metric $du^{2}$ on $\Sigma$ given by (\ref{equ-proof-221}), 
which has strictly negative curvature on ${\Sigma}'$. 

Take a point $p\in {\Sigma}'$. Since the metric $d{\sigma}^{2}$ is flat on ${\Sigma}'$, by Lemma \ref{main-lem2}, 
there exists a local isometry $\Phi$ satisfying $\Phi (0)=p$ from a disk ${\triangle}_{R}=\{z\in \C\,;\, |z|<R \}$ $(0<R \leq +\infty)$ with 
the standard metric $ds_{Euc}^{2}$ onto an open neighborhood of $p$ in ${\Sigma}'$ with the metric $d{\sigma}^{2}$. 
We shall denote the functions $g\circ \Phi$ and $\hat{g}\circ \Phi (= \hat{g}\circ \Psi \circ \Phi)$ by, respectively, 
$g$ and $\hat{g}$ in the following. 
Moreover the pseudo-metric ${\Phi}^{\ast}du^{2}$ on ${\triangle}_{R}$ has strictly negative curvature. Since there exists no 
metric with strictly negative curvature on $\C$ (see \cite[Corollary 4.2.4]{Fu1993}), we get that $R< +\infty$. 
Furthermore, by the same argument as in the proof of Theorem \ref{thm-ramification}, we can choose a point $a_{0}$ with $|a_{0}|=1$ such that, 
for the curve $L_{a_{0}}=\{w:=a_{0}s\,;\, 0<s<R \}$, the $\Phi$-image ${\Gamma}_{a_{0}}$ tends to the boundary of ${\Sigma}'$ as $s$ tends to $R$. 
Here, if we suitably choose the constant $\eta$ in the definition (\ref{equ-proof-224}) of $\lambda$, 
then ${\Gamma}_{a_{0}}$ tends to the boundary of $\Sigma$. 

Since ${\Phi}^{\ast}d{\sigma}^{2}=|dz|^{2}$, we get by (\ref{equ-proof-225}) that 
$$
|h_{z}|^{2}=\biggl{(}\dfrac{|g-\hat{g}|^{2}|g'_{z}||\hat{g}'_{z}|\prod_{j=1}^{q-1}(1+|{\alpha}_{j}|^{2})^{(1-\eta)/2}}{\prod_{j=1}^{q-1}(|g-{\alpha}_{j}||\hat{g}-{\alpha}_{j}|)^{1-\eta}} \biggr{)}^{\lambda}. 
$$
By (\ref{equ-proof-223}), we have that 
\begin{eqnarray}
ds^{2} &=& |h_{z}|^{2}(1+|g|^{2})^{m/2}(1+|\hat{g}|^{2})^{m/2}|dz|^{2} \nonumber \\
       &=& \biggl{(}\dfrac{|g-\hat{g}||g'_{z}||\hat{g}'_{z}|(1+|g|^{2})^{m/2\lambda}(1+|\hat{g}|^{2})^{m/2\lambda}\prod_{j=1}^{q-1}(1+|{\alpha}_{j}|^{2})^{(1-\eta)/2}}{\prod_{j=1}^{q-1}(|g-{\alpha}_{j}||\hat{g}-{\alpha}_{j}|)^{1-\eta}} \biggr{)}^{\lambda}|dz|^{2} \nonumber \\
       &=& \Biggl{(} {\mu}^{2}\prod_{j=1}^{q}(|g-{\alpha}_{j}||\hat{g}-{\alpha}_{j}|)^{\varepsilon}\biggl{(}\log{\dfrac{b_{0}}{|g, {\alpha}_{j}|^{2}}}\log{\dfrac{b_{0}}{|\hat{g}, {\alpha}_{j}|^{2}}}\biggr{)}^{1-\varepsilon} \Biggr{)}^{\lambda}|dz|^{2}, \nonumber
\end{eqnarray}
where $\mu$ is the function with $du^{2}={\mu}^{2}|dz|^{2}$. Since the function $x^{\varepsilon}\log^{1-\varepsilon}(b_{0}/x^{2})$ $(0< x\leq 1)$ is 
bounded, we obtain that 
$$
ds^{2}\leq C''\biggl{(}\dfrac{|g, \hat{g}|^{2}|g'_{z}||\hat{g}'_{z}|\xi\hat{\xi}}{(1+|g|^{2})(1+|\hat{g}|^{2})} \biggr{)}^{\lambda}|dz|^{2}
$$
for some constant $C''$. By Lemma \ref{main-lem4}, we have that 
$$
ds^{2}\leq C'''\biggl{(}\dfrac{R}{R-|z|^{2}} \biggr{)}^{\lambda}|dz|^{2}
$$
for some constant $C'''$. Thus we obtain that 
$$
\int_{{\Gamma}_{a_{0}}}ds \leq (C''')^{1/2}\int_{L_{a_{0}}}\biggl{(}\dfrac{R}{R^{2}-|z|^{2}} \biggr{)}^{\lambda /2}|dz|\leq C\dfrac{R^{(2-\lambda)/2}}{1-(\lambda /2)}< +\infty
$$
because $0< \lambda < 1$. However it contradicts the assumption that the metric $ds^{2}$ is complete. 
Hence we have necessarily $g\equiv \hat{g}$. 
\end{proof}

\section{Applications} 
In this section, we give several applications of our main results. 

\subsection{Gauss map of minimal surfaces in ${\R}^{3}$} 

We first recall some basic facts of minimal surfaces in Euclidean 3-space ${\R}^{3}$. Details can be found, for example, 
in \cite{Fu1993}, \cite{La1982} and \cite{Os1986}. Let $X=(x^{1}, x^{2}, x^{3})\colon \Sigma \to {\R}^{3}$ be an oriented minimal surface in ${\R}^{3}$. 
By associating a local complex coordinate $z=u+\sqrt{-1}v$ with each positive isothermal coordinate system $(u, v)$, $\Sigma$ is considered as 
a Riemann surface whose conformal metric is the induced metric $ds^{2}$ from ${\R}^{3}$. Then 
\begin{equation}\label{equ-411-Laplacian}
{\triangle}_{ds^{2}}X=0
\end{equation}
holds, that is, each coordinate function $x^{i}$ is harmonic. With respect to the local complex coordinate $z=u+\sqrt{-1}v$ of the surface, 
(\ref{equ-411-Laplacian}) is given by 
$$
\bar{\partial}\partial X = 0, 
$$
where $\partial =(\partial /\partial u -\sqrt{-1}\partial /\partial v)/2$, $\bar{\partial} =(\partial /\partial u +\sqrt{-1}\partial /\partial v)/2$. 
Hence each ${\phi}_{i}:=\partial x^{i}dz$ $(i=1, 2, 3)$ is a holomorphic 1-form on $\Sigma$. If we set that 
$$
\omega ={\phi}_{1}-\sqrt{-1}{\phi}_{2}, \quad g=\dfrac{{\phi}_{3}}{{\phi}_{1}-\sqrt{-1}{\phi}_{2}}, 
$$
then $\omega$ is a holomorphic 1-form and $g$ is a meromorphic function on $\Sigma$. Moreover the function $g$ coincides with the composition of 
the Gauss map and the stereographic projection from ${\mathbf{S}}^{2}$ onto $\C\cup \{\infty \}$, and the induced metric is given by 
\begin{equation}\label{equ-412-metric}
ds^{2}=(1+|g|^{2})^{2}|\omega|^{2}. 
\end{equation}

Applying Theorem \ref{thm-ramification} to the metric $ds^{2}$, we can get the Fujimoto curvature bound 
for a minimal surface in ${\R}^{3}$. 

\begin{theorem}[{\cite[Theorem C]{Fu1992}}, {\cite[Theorem 1.6.1]{Fu1993}}] \label{thm-4-minimal1}
Let $X\colon \Sigma \to {\R}^{3}$ be an oriented minimal surface. 
Let $q\in \N$, ${\alpha}_{1}, \ldots, {\alpha}_{q}\in \C\cup \{\infty \}$ be distinct and ${\nu}_{1}, \ldots, {\nu}_{q}\in \N\cup \{\infty \}$. 
Suppose that 
\begin{equation}\label{equ-413-ramification} 
\gamma= \displaystyle \sum_{j=1}^{q} \biggl{(}1-\dfrac{1}{{\nu}_{j}} \biggr{)}> 4\,(=2+2). 
\end{equation}
If the Gauss map $g\colon \Sigma \to \C\cup \{\infty \}$ satisfies the property that 
all ${\alpha}_{j}$-points of $g$ have multiplicity at least ${\nu}_{j}$, then there exists a positive constant $C$, depending on 
${\alpha}_{1}, \ldots, {\alpha}_{q}$, but not the surface, such that for all $p\in \Sigma$ the inequality (\ref{equ-curvature}) holds. 
\end{theorem}

As a corollary of Theorem \ref{thm-4-minimal1}, we have the following ramification theorem for the Gauss map of a complete minimal surface in ${\R}^{3}$. 

\begin{corollary}[{\cite[Theorem 2]{Ru1993}}, {\cite[Corollary 3.4]{Fu1988}}]\label{thm-4-minimal2}
Let $X\colon \Sigma \to {\R}^{3}$ be a complete minimal surface. 
Let $q\in \N$, ${\alpha}_{1}, \ldots, {\alpha}_{q}\in \C\cup \{\infty \}$ be distinct and ${\nu}_{1}, \ldots, {\nu}_{q}\in \N\cup \{\infty \}$. 
Suppose that the inequality (\ref{equ-413-ramification}) holds. If the Gauss map $g\colon \Sigma \to \C\cup \{\infty \}$ satisfies the property that 
all ${\alpha}_{j}$-points of $g$ have multiplicity at least ${\nu}_{j}$, then $g$ must be constant, that is, $X(\Sigma)$ is a plane. 
In particular, the Gauss map of a nonflat complete minimal surface in ${\R}^{3}$ can omit at most $4\,(=2+2)$ values. 
\end{corollary}

We remark that the author, Kobayashi and Miyaoka \cite{KKM2008} gave a similar result for the Gauss map of 
a special class of complete minimal surfaces (this class is called the pseudo-algebraic minimal surfaces). 

As an application of Corollary \ref{thm-4-minimal2}, we obtain the following analogue to the Ahlfors islands theorem. 
We remark that Klots and Sario \cite{KS1966} investigated the upper bound for the number of islands for the Gauss map of a minimal surface in ${\R}^{3}$. 

\begin{theorem}\label{thm-4-minimal3}
Let $X\colon \Sigma \to {\R}^{3}$ be a complete minimal surface. 
Let $q\in \N$, ${\alpha}_{1}, \ldots, {\alpha}_{q}\in \C\cup \{\infty \}$ be distinct, 
$D_{j}({\alpha}_{j}, \varepsilon):=\{z\in \C\cup \{\infty \} \,;\, |z, {\alpha}_{j}|< \varepsilon \}$ $(1\leq j\leq q)$ be pairwise disjoint and 
${\nu}_{1}, \ldots, {\nu}_{q}\in \N$. Suppose that the inequality (\ref{equ-413-ramification}) holds. 
Then there exists $\varepsilon > 0$ such that, if the Gauss map $g$ has no island of multiplicity less 
than ${\nu}_{j}$ over $D_{j}({\alpha}_{j}, \varepsilon)$ 
for all $j\in \{1, \ldots , q\}$, then $g$ must be constant, that is, $X(\Sigma)$ is a plane. 
\end{theorem}

The important case of Theorem \ref{thm-4-minimal3} is the case where $q=9\,(=2\cdot 2+5)$ and ${\nu}_{j}=2$ for each $j$ 
$(j=1,\ldots, q)$. 

\begin{corollary}\label{thm-4-minimal4}
Let $X\colon \Sigma \to {\R}^{3}$ be a complete minimal surface. Let ${\alpha}_{1}, \ldots, 
{\alpha}_{9} \in \C\cup \{\infty \}$ be distinct and 
$D_{j}({\alpha}_{j}, \varepsilon):=\{z\in \C\cup \{\infty \} \,;\, |z, {\alpha}_{j}|< \varepsilon \}$ $(1\leq j\leq 9)$. 
Then there exists $\varepsilon > 0$ such that, if the Gauss map $g$ has no simple island of over any of the small disks $D_{j}({\alpha}_{j}, \varepsilon)$ 
for all $j\in \{1, \ldots , 9 \}$, then $g$ must be constant, that is, $X(\Sigma)$ is a plane. 
\end{corollary}

Finally, by applying Theorem \ref{thm-unicity}, we provide the Fujimoto unicity theorem for the Gauss maps of complete minimal surfaces in ${\R}^{3}$. 

\begin{theorem}[{\cite[Theorem I]{Fu1993-2}}]\label{thm-4-minimal5}
Let $X\colon \Sigma \to {\R}^{3}$ and $\widehat{X}\colon \widehat{\Sigma}\to {\R}^{3}$ be two nonflat minimal surfaces and 
assume that there exists a conformal diffeomorphism $\Psi\colon \Sigma \to \widehat{\Sigma}$. Let $g\colon \Sigma \to \C\cup\{\infty \}$ and 
$\hat{g}\colon \widehat{\Sigma}\to \C\cup\{\infty \}$ be the Gauss maps of $X(\Sigma)$ and $\widehat{X}(\widehat{\Sigma})$, respectively. 
If $g\not\equiv \hat{g}\circ \Psi$ and either $X(\Sigma)$ or $\widehat{X}(\widehat{\Sigma})$ is complete, 
then $g$ and $\hat{g}\circ \Psi$ share at most $6\,(=2+4)$ distinct values. 
\end{theorem}

\subsection{Hyperbolic Gauss map of constant mean curvature one surfaces in ${\H}^{3}$} 
We denote by ${\H}^{3}$ hyperbolic 3-space, that is, the simply connected Riemannian 3-manifold with constant sectional curvature $-1$, 
which is represented as 
$$
{\H}^{3}= SL(2, \C) / SU(2) =\{aa^{\ast} \,;\, a\in SL(2, \C) \} \quad (a^{\ast}:=\tr{\bar{a}}). 
$$
Then there exists the following representation formula for constant mean curvature one (CMC-1, for short) surfaces in ${\H}^{3}$ 
as an analogy of the Enneper-Weierstrass representation formula in minimal surface theory.  

\begin{theorem}[\cite{Br1987}, \cite{UY1993}]\label{thm-CMC-1}
Let $\widetilde{\Sigma}$ be a simply connected Riemann surface with a base point $z_{0}\in \widetilde{\Sigma}$ and $(g, \omega)$ a pair consisting 
of a meromorphic function and a holomorphic 1-form on $\widetilde{\Sigma}$ such that 
\begin{equation}\label{equ-CMC-1}
ds^{2}=(1+|g|^{2})^{2}|\omega|^{2}
\end{equation}
gives a (positive definite) Riemannian metric on $\widetilde{\Sigma}$. Take a holomorphic immersion $F=(F_{ij})\colon \widetilde{\Sigma}\to SL(2, \C)$ satisfying 
$F(z_{0})=id$ and 
\begin{equation}\label{equ-CMC-2}
F^{-1}dF=\left(
\begin{array}{cc}
g & -g^{2}  \\
1 & -g
\end{array}
\right)\omega.
\end{equation}
Then $f\colon \widetilde{\Sigma}\to {\H}^{3}$ defined by 
\begin{equation}\label{equ-CMC-3}
f=FF^{\ast}
\end{equation}
is a CMC-1 surface and the induced metric of $f$ is $ds^{2}$. Moreover the second fundamental form $h$ and the Hopf differential $Q$ of $f$ are 
given by 
$$
h=-Q-\overline{Q}+ds^{2}, \quad Q=\omega\, dg. 
$$ 
Conversely, for any CMC-1 surface $f\colon \widetilde{\Sigma}\to {\H}^{3}$, there exist a meromorphic function $g$ and a holomorphic 1-form $\omega$ 
on $\widetilde{\Sigma}$ such that the induced metric of $f$ is given by (\ref{equ-CMC-1}) and (\ref{equ-CMC-3}) holds, where the map 
$F\colon \widetilde{\Sigma}\to SL(2, \C)$ is a holomorphic null (``null'' means $\det{(F^{-1}dF)}=0$) immersion satisfying (\ref{equ-CMC-2}). 
\end{theorem}

Following the terminology of \cite{UY1993}, $g$ is called a {\it secondary Gauss map} of $f$. The pair $(g, \omega)$ is called {\it Weierstrass data} 
of $f$. Let $f\colon \Sigma\to {\H}^{3}$ be a CMC-1 surface on a (not necessarily simply connected) Riemann surface $\Sigma$. Then the map $F$ 
is defined only on its universal covering surface $\widetilde{\Sigma}$. Thus the pair $(\omega, g)$ is not single-valued on $\Sigma$. 
However the {\it hyperbolic Gauss map} of $f$ defined by 
$$
G=\dfrac{dF_{11}}{dF_{21}}=\dfrac{dF_{12}}{dF_{22}}, \quad \text{where}\quad F(z)=\left(
\begin{array}{cc}
F_{11}(z) & F_{12}(z)  \\
F_{21}(z) & F_{22}(z)
\end{array}
\right)
$$ 
is a single-valued meromorphic function on $\Sigma$. By identifying the ideal boundary ${\Si}^{2}_{\infty}$ of ${\H}^{3}$ with the Riemann 
sphere $\C\cup\{\infty \}$, the geometric meaning of $G$ is given as follows (cf. \cite{Br1987}): The hyperbolic Gauss map $G$ sends each 
$p\in \Sigma$ to the point $G(p)$ at ${\Si}^{2}_{\infty}$ reached by the oriented normal geodesics emanating from the surface.  The inverse 
matrix $F^{-1}$ is also a holomorphic null immersion and produce a new CMC-1 surface $f^{\sharp}:=F^{-1}(F^{-1})^{\ast}\colon \widetilde{\Sigma} \to 
{\H}^{3}$ which is called the {\it dual} of $f$ (\cite{UY1997}). Then the Weierstrass data $(g^{\sharp}, {\omega}^{\sharp})$, the Hopf differential 
$Q^{\sharp}$, and the hyperbolic Gauss map $G^{\sharp}$ of $f^{\sharp}$ are given by 
\begin{equation}\label{equ-CMC-4}
g^{\sharp}=G, \quad {\omega}^{\sharp}=-\dfrac{Q}{dG}, \quad Q^{\sharp}=-Q, \quad G^{\sharp}= g. 
\end{equation}
By Theorem \ref{thm-CMC-1} and (\ref{equ-CMC-4}) , the induced metric $ds^{2\sharp}$ of $f^{\sharp}$ is given by 
\begin{equation}\label{equ-CMC-5}
ds^{2\sharp}=(1+|g^{\sharp}|^{2})^{2}|{\omega}^{\sharp}|^{2}=(1+|G|^{2})^{2}\biggl{|}\dfrac{Q}{dG}\biggr{|}^{2}. 
\end{equation}
We call the metric $ds^{2\sharp}$ the {\it dual metric} of $f$. There exists the following relationship between the metric $ds^{2}$ and 
the dual metric $ds^{2\sharp}$. 

\begin{theorem}[\cite{UY1997}, \cite{Yu1997}]\label{thm-CMC-2}
The metric $ds^{2}$ is complete (resp. nondegenerate) if and only if the dual metric $ds^{2\sharp}$ is complete (resp. nondegenerate). 
\end{theorem}

Applying Theorem \ref{thm-ramification} to the dual metric $ds^{2\sharp}$, we get the following theorem. 
\begin{theorem}\label{thm-CMC-3} 
Let $f\colon \Sigma \to {\H}^{3}$ be a CMC-1 surface. Let $q\in \N$, ${\alpha}_{1}, \ldots, {\alpha}_{q}\in \C\cup \{\infty \}$ be 
distinct and ${\nu}_{1}, \ldots, {\nu}_{q}\in \N\cup \{\infty \}$. Suppose that the inequality (\ref{equ-413-ramification}) holds. 
If the hyperbolic Gauss map $G\colon \Sigma \to \C\cup \{\infty \}$ satisfies the property that all ${\alpha}_{j}$-points of $G$ have 
multiplicity at least ${\nu}_{j}$, then there exists a positive constant $C$, depending on 
${\alpha}_{1}, \ldots, {\alpha}_{q}$, but not the surface, such that for all $p\in \Sigma$ we have
$$
|K_{ds^{2\sharp}}(p)|^{1/2}\leq \dfrac{C}{d(p)}, 
$$
where $K_{ds^{2\sharp}}(p)$ is the Gaussian curvature of the metric $ds^{2\sharp}$ at $p$ and $d(p)$ is the geodesic distance from $p$ to 
the boundary of $\Sigma$. 
\end{theorem}

Combining of Theorems \ref{thm-CMC-2} and \ref{thm-CMC-3}, we get the following ramification theorem for the hyperbolic Gauss map of CMC-1 surfaces 
in ${\H}^{3}$. 
\begin{corollary}[\cite{Ka2011}, \cite{Yu1997}]\label{thm-CMC-4}
Let $f\colon \Sigma \to {\H}^{3}$ be a complete CMC-1 surface. 
Let $q\in \N$, ${\alpha}_{1}, \ldots, {\alpha}_{q}\in \C\cup \{\infty \}$ be distinct and ${\nu}_{1}, \ldots, {\nu}_{q}\in \N\cup \{\infty \}$. 
Suppose that the inequality (\ref{equ-413-ramification}) holds. If the hyperbolic Gauss map $G\colon \Sigma \to \C\cup \{\infty \}$ satisfies 
the property that all ${\alpha}_{j}$-points of $G$ have multiplicity at least ${\nu}_{j}$, then $G$ must be constant, that is, $f(\Sigma)$ is a horosphere. 
In particular, if the hyperbolic Gauss map of a nonflat complete CMC-1 surface in ${\H}^{3}$ can omit at most $4\,(=2+2)$ values. 
\end{corollary}

Moreover we obtain the following analogue to the Ahlfors islands theorem. 

\begin{theorem}\label{thm-CMC-5}
Let $f\colon \Sigma \to {\H}^{3}$ be a complete CMC-1 surface. 
Let $q\in \N$, ${\alpha}_{1}, \ldots, {\alpha}_{q}\in \C\cup \{\infty \}$ be distinct, 
$D_{j}({\alpha}_{j}, \varepsilon):=\{z\in \C\cup \{\infty \} \,;\, |z, {\alpha}_{j}|< \varepsilon \}$ $(1\leq j\leq q)$ be pairwise disjoint and 
${\nu}_{1}, \ldots, {\nu}_{q}\in \N$. Suppose that the inequality (\ref{equ-413-ramification}) holds. 
Then there exists $\varepsilon > 0$ such that, if the hyperbolic Gauss map $G$ has no island of multiplicity less than ${\nu}_{j}$ over 
$D_{j}({\alpha}_{j}, \varepsilon)$ for all $j\in \{1, \ldots , q\}$, then $G$ must be constant, that is, $f(\Sigma)$ is a horosphere. 
\end{theorem}

The important case of Theorem \ref{thm-CMC-5} is the case where $q=9\,(=2\cdot 2+5)$ and ${\nu}_{j}=2$ for each $j$ 
$(j=1, \ldots, q)$. 

\begin{corollary}\label{thm-CMC-6}
Let $f\colon \Sigma \to {\H}^{3}$ be a complete CMC-1 surface. Let ${\alpha}_{1}, \ldots, {\alpha}_{9} \in \C\cup \{\infty \}$ be distinct and 
$D_{j}({\alpha}_{j}, \varepsilon):=\{z\in \C\cup \{\infty \} \,;\, |z, {\alpha}_{j}|< \varepsilon \}$ $(1\leq j\leq 9)$. 
Then there exists $\varepsilon > 0$ such that, if the hyperbolic Gauss map $G$ has no simple island of over any of the small disks 
$D_{j}({\alpha}_{j}, \varepsilon)$ for all $j\in \{1, \ldots , 9 \}$, then $G$ must be constant, that is, $f(\Sigma)$ is a horosphere. 
\end{corollary}

Finally, by applying Theorem \ref{thm-unicity}, we provide the following unicity theorem for the hyperbolic Gauss maps of complete 
CMC-1 surfaces in ${\H}^{3}$. 

\begin{theorem}\label{thm-CMC-7}
Let $f\colon \Sigma \to {\H}^{3}$ and $\widehat{f}\colon \widehat{\Sigma}\to {\H}^{3}$ be two nonflat CMC-1 surfaces and 
assume that there exists a conformal diffeomorphism $\Psi\colon \Sigma \to \widehat{\Sigma}$. Let $G\colon \Sigma \to \C\cup\{\infty \}$ and 
$\widehat{G}\colon \widehat{\Sigma}\to \C\cup\{\infty \}$ be the hyperbolic Gauss maps of $f(\Sigma)$ and $\widehat{f}(\widehat{\Sigma})$, respectively. 
If $G\not\equiv \hat{G}\circ \Psi$ and either $f(\Sigma)$ or $\hat{f}(\widehat{\Sigma})$ is complete, 
then $G$ and $\hat{G}\circ \Psi$ share at most $6\,(=2+4)$ distinct values. 
\end{theorem}

\subsection{Lorentzian Gauss map of maxfaces in ${\R}^{3}_{1}$} 
Maximal surfaces in the  Lorentz-Minkowski 3-space ${\R}^{3}_{1}$ are closely related to minimal surfaces in ${\R}^{3}$. 
We treat maximal surfaces with some admissible singularities, called {\it maxfaces}, as introduced by Umehara and Yamada \cite{UY2006}. 
We remark that maxfaces, 
non-branched generalized maximal surfaces in the sense of \cite{ER1992} and non-branched generalized maximal maps in the sense of \cite{IK2008} are 
all the same class of maximal surfaces. The Lorentz-Minkowski $3$-space ${\R}^{3}_{1}$ is the affine $3$-space ${\R}^{3}$ with the inner product 
$$
\langle \, , \, \rangle = -(dx^{1})^{2}+(dx^{2})^{2}+(dx^{3})^{2},
$$
where $(x^{1}, x^{2}, x^{3})$ is the canonical coordinate system of ${\R}^{3}$. We consider a fibration 
$$
p_{L}\colon {\C}^{3} \ni ({\zeta}^{1}, {\zeta}^{2}, {\zeta}^{3}) \mapsto \text{Re} (-\sqrt{-1}{\zeta}^{1}, {\zeta}^{2}, {\zeta}^{3}) \in {\R}^{3}_{1}. 
$$
The projection of null holomorphic immersions into  ${\R}^{3}_{1}$ by $p_{L}$ gives maxfaces. Here, a holomorphic map $F=(F_{1}, F_{2}, F_{3})\colon \Sigma \to {\C}^{3}$ 
is said to be {\it null} if $\{(F_{1})'_{z}\}^{2}+\{(F_{2})'_{z}\}^{2}+\{(F_{3})'_{z}\}^{2}$ vanishes identically, where $'=d / dz$ denotes the derivative with respect to 
a local complex coordinate $z$ of $\Sigma$. For a maxface, an analogue of the Enneper-Weierstrass representation formula is known (see also \cite{Ko1983}). 

\begin{theorem}[{\cite[Theorem 2.6]{UY2006}}]\label{max-EW}
Let $\Sigma$ be a Riemann surface and $(g, \omega)$ a pair consisting of a meromorphic function and a holomorphic $1$-form on $\Sigma$ such that 
\begin{equation}\label{max-nullmet}
d{\sigma}^{2}:= (1+|g|^{2})^{2}|\omega|^{2}
\end{equation}
gives a (positive definite) Riemannian metric on $\Sigma$, and $|g|$ is not identically $1$. Assume that 
$$
\text{Re} \int_{\gamma} (-2g, 1+g^{2}, \sqrt{-1}(1-g^{2}))\, \omega = 0
$$
for all loops $\gamma$ in $\Sigma$. Then 
\begin{equation}\label{max-immer}
f= \text{Re} \int^{z}_{z_{0}} (-2g, 1+g^{2}, \sqrt{-1}(1-g^{2}))\, \omega
\end{equation}
is well-defined on $\Sigma$ and gives a maxface in ${\R}^{3}_{1}$, where $z_{0}\in {\Sigma}$ is a base point. Moreover, all maxfaces are obtained 
in this manner. The induced metric $ds^{2}:=f^{\ast} \langle \, , \, \rangle$ is given by
$$
ds^{2} = (1-|g|^{2})^{2}|\omega|^{2}, 
$$
and the point $p\in \Sigma$ is a singular point of $f$ if and only if $|g(p)|=1$. 
\end{theorem}

We call $g$ the {\it Lorentzian Gauss map} of $f$. If $f$ has no singularities, then $g$ coincides with the composition of the 
Gauss map (i.e. (Lorentzian) unit normal vector) $n \colon \Sigma \to {\H}^{2}_{\pm}$ into the upper or lower connected component of 
the two-sheet hyperboloid ${\H}^{2}_{\pm}={\H}^{2}_{+}\cup {\H}^{2}_{-}$ in ${\R}^{3}_{1}$, where 
\begin{eqnarray}
{\H}^{2}_{+} &:=& \{n=(n^{1}, n^{2}, n^{3})\in {\R}^{3}_{1}\, ; \, \langle n, n \rangle = -1, n^{1}> 0\}, \nonumber \\ 
{\H}^{2}_{-} &:=& \{n=(n^{1}, n^{2}, n^{3})\in {\R}^{3}_{1}\, ; \, \langle n, n \rangle = -1, n^{1}< 0\},   \nonumber
\end{eqnarray}
and the stereographic projection from the north pole $(1, 0, 0)$ of the hyperboloid onto the Riemann sphere $\C\cup \{\infty\}$ 
(see \cite[Section 1]{UY2006} for the details). 
A maxface is said to be {\it weakly complete} if the metric $d{\sigma}^{2}$ as in (\ref{max-nullmet}) is complete. 
We note that $(1/2)d{\sigma}^{2}$ coincides with 
the pull-back of the standard metric on ${\C}^{3}$ by the null holomorphic immersion of $f$ (see \cite[Section 2]{UY2006}). 

Applying Theorem \ref{thm-ramification} to the metric $d{\sigma}^{2}$, we can get the following curvature estimate. 

\begin{theorem}\label{thm-4-maximal1}
Let $f\colon \Sigma \to {\R}^{3}_{1}$ be a maxface. 
Let $q\in \N$, ${\alpha}_{1}, \ldots, {\alpha}_{q}\in \C\cup \{\infty \}$ be distinct and ${\nu}_{1}, \ldots, {\nu}_{q}\in \N\cup \{\infty \}$. 
Suppose that the inequality (\ref{equ-413-ramification}) holds. 
If the Lorentzian Gauss map $g\colon \Sigma \to \C\cup \{\infty \}$ satisfies the property that 
all ${\alpha}_{j}$-points of $g$ have multiplicity at least ${\nu}_{j}$, then there exists a positive constant $C$, depending on 
${\alpha}_{1}, \ldots, {\alpha}_{q}$, but not the surface, such that for all $p\in \Sigma$ we have
$$
|K_{d{\sigma}^{2}}(p)|^{1/2}\leq \dfrac{C}{d(p)}, 
$$
where $K_{d{\sigma}^{2}}(p)$ is the Gaussian curvature of the metric $d{\sigma}^{2}$ at $p$ and $d(p)$ is the geodesic distance from $p$ to 
the boundary of $\Sigma$. 
\end{theorem}

As a corollary of Theorem \ref{thm-4-maximal1}, 
we have the following ramification theorem for the Lorentzian Gauss map of a weakly complete maxface in ${\R}^{3}_{1}$.

\begin{corollary}\label{thm-4-maximal2}
Let $f\colon \Sigma \to {\R}^{3}_{1}$ be a weakly complete maxface. 
Let $q\in \N$, ${\alpha}_{1}, \ldots, {\alpha}_{q}\in \C\cup \{\infty \}$ be distinct and ${\nu}_{1}, \ldots, {\nu}_{q}\in \N\cup \{\infty \}$. 
Suppose that the inequality (\ref{equ-413-ramification}) holds. If the Lorentzian Gauss map $g\colon \Sigma \to \C\cup \{\infty \}$ satisfies 
the property that all ${\alpha}_{j}$-points of $g$ have multiplicity at least ${\nu}_{j}$, then $g$ must be constant, that is, $f(\Sigma)$ is a plane. 
In particular, the Lorentzian Gauss map of a nonflat weakly complete maxface in ${\R}^{3}_{1}$ can omit at most $4\,(=2+2)$ values. 
\end{corollary}

As an application of Corollary \ref{thm-4-maximal2}, we can give a simple proof of the Calabi-Bernstein theorem (\cite{Ca1970}, \cite{CY1976}) 
for a maximal space-like surface in ${\R}^{3}_{1}$ from the viewpoint of function-theoretic properties of the Lorenzian Gauss map. 
For the details, see \cite[Section 4.2]{Ka2013}. 

As another application of Corollary \ref{thm-4-maximal2}, we obtain the following analogue to the Ahlfors islands theorem. 

\begin{theorem}\label{thm-4-maximal3}
Let $f\colon \Sigma \to {\R}^{3}_{1}$ be a weakly complete maxface. 
Let $q\in \N$, ${\alpha}_{1}, \ldots, {\alpha}_{q}\in \C\cup \{\infty \}$ be distinct, 
$D_{j}({\alpha}_{j}, \varepsilon):=\{z\in \C\cup \{\infty \} \,;\, |z, {\alpha}_{j}|< \varepsilon \}$ $(1\leq j\leq q)$ be pairwise disjoint and 
${\nu}_{1}, \ldots, {\nu}_{q}\in \N$. Suppose that the inequality (\ref{equ-413-ramification}) holds. 
Then there exists $\varepsilon > 0$ such that, if the Lorentzian Gauss map $g$ has no island of multiplicity less than ${\nu}_{j}$ over $D_{j}({\alpha}_{j}, \varepsilon)$ 
for all $j\in \{1, \ldots , q\}$, then $g$ must be constant, that is, $f(\Sigma)$ is a plane. 
\end{theorem}

The important case of Theorem \ref{thm-4-maximal3} is the case where $q=9\,(=2\cdot 2+5)$ and ${\nu}_{j}=2$ for each $j$ 
$(j=1,\ldots, q)$. 

\begin{corollary}\label{thm-4-maximal4}
Let $f\colon \Sigma \to {\R}^{3}_{1}$ be a weakly complete maxface. Let ${\alpha}_{1}, \ldots, {\alpha}_{9} \in \C\cup \{\infty \}$ be distinct and 
$D_{j}({\alpha}_{j}, \varepsilon):=\{z\in \C\cup \{\infty \} \,;\, |z, {\alpha}_{j}|< \varepsilon \}$ $(1\leq j\leq 9)$. 
Then there exists $\varepsilon > 0$ such that, if the Lorentzian Gauss map $g$ has no simple island of over any of the small disks 
$D_{j}({\alpha}_{j}, \varepsilon)$ for all $j\in \{1, \ldots , 9 \}$, then $g$ must be constant, that is, $f(\Sigma)$ is a plane. 
\end{corollary}

Finally, by applying Theorem \ref{thm-unicity}, we provide the following unicity theorem for the Lorentzian Gauss maps of weakly complete 
maxfaces in ${\R}^{3}_{1}$. 

\begin{theorem}\label{thm-4-maximal5}
Let $f\colon \Sigma \to {\R}^{3}_{1}$ and $\widehat{f}\colon \widehat{\Sigma}\to {\R}^{3}_{1}$ be two nonflat maxfaces and 
assume that there exists a conformal diffeomorphism $\Psi\colon \Sigma \to \widehat{\Sigma}$. Let $g\colon \Sigma \to \C\cup\{\infty \}$ and 
$\hat{g}\colon \widehat{\Sigma}\to \C\cup\{\infty \}$ be the Lorentzian Gauss maps of $f(\Sigma)$ and $\widehat{f}(\widehat{\Sigma})$, respectively. 
If $g\not\equiv \hat{g}\circ \Psi$ and either $f(\Sigma)$ or $\widehat{f}(\widehat{\Sigma})$ is weakly complete, 
then $g$ and $\hat{g}\circ \Psi$ share at most $6\,(=2+4)$ distinct values. 
\end{theorem}

\subsection{Lagrangian Gauss map of improper affine fronts in ${\R}^{3}$} 
Improper affine spheres in the affine $3$-space ${\R}^{3}$ also have similar properties to minimal surfaces in Euclidean $3$-space. 
Recently, Mart\'inez \cite{Ma2005} discovered the correspondence between improper affine spheres and smooth special Lagrangian immersions in 
the complex $2$-space ${\C}^{2}$ and introduced the notion of {\it improper affine fronts}, that is, a class of (locally strongly convex) improper 
affine spheres with some admissible singularities in ${\R}^{3}$. We remark that this class is called 
``improper affine maps'' in \cite{Ma2005}, but we call this class ``improper affine fronts'' because all of improper affine maps are 
wave fronts in ${\R}^{3}$ (\cite{Na2009}, \cite{UY2011}). 
The differential geometry of wave fronts is discussed in \cite{SUY2009}. 
Moreover, Mart\'inez gave the following holomorphic representation for this class. 

\begin{theorem}[{\cite[Theorem 3]{Ma2005}}]
Let $\Sigma$ be a Riemann surface and $(F, G)$ a pair of holomorphic functions on $\Sigma$ such that $\text{Re}(FdG)$ is exact and 
$|dF|^{2}+|dG|^{2}$ is positive definite. Then the induced map $\psi\colon \Sigma \to {\R}^{3}=\C\times {\R}$ given by 
$$
\psi :=\biggl{(}G+\overline{F}, \dfrac{|G|^{2}-|F|^{2}}{2}+\text{Re}\biggl{(} GF- 2\int FdG \biggr{)} \biggr{)}
$$
is an improper affine front. Conversely, any improper affine front is given in this way. 
Moreover we set $x:= G+\overline{F}$ and $n:= \overline{F}-G$. Then $L_{\psi}:=x+\sqrt{-1}n\colon \Sigma \to {\C}^{2}$ is a special Lagrangian 
immersion whose induced metric $d{\tau}^{2}$ from ${\C}^{2}$ is given by 
$$
d{\tau}^{2}=2(|dF|^{2}+|dG|^{2}). 
$$
In addition, the affine metric $h$ of $\psi$ is expressed as $h:=|dG|^{2}-|dF|^{2}$ and the singular points of $\psi$ correspond to the 
points where $|dF|=|dG|$. 
\end{theorem}

We remark that Nakajo \cite{Na2009} constructed a representation formula for indefinite improper affine spheres with some admissible singularities. 

The nontrivial part of the Gauss map of $L_{\psi}\colon \Sigma \to {\C}^{2}\simeq {\R}^{4}$ (see \cite{CM1987}) is the meromorphic function 
$\nu\colon \Sigma \to \C\cup\{\infty \}$ given by 
$$
\nu := \dfrac{dF}{dG}, 
$$
which is called the {\it Lagrangian Gauss map} of $\psi$. An improper affine front is said to be {\it weakly complete} if the induced metric 
$d{\tau}^{2}$ is complete. We remark that 
$$
d{\tau}^{2}=2(|dF|^{2}+|dG|^{2})=2(1+|\nu|^{2})|dG|^{2}. 
$$

Applying Theorem \ref{thm-ramification} to the metric $d{\tau}^{2}$, we can get the following theorem. 
This is a generalization of \cite[Theorem 4.6]{Ka2013}. 

\begin{theorem}\label{thm-4-improper1}
Let $\psi\colon \Sigma \to {\R}^{3}$ be an improper affine front. 
Let $q\in \N$, ${\alpha}_{1}, \ldots, {\alpha}_{q}\in \C\cup \{\infty \}$ be distinct and ${\nu}_{1}, \ldots, {\nu}_{q}\in \N\cup \{\infty \}$. 
Suppose that 
\begin{equation}\label{equ-4-imramification} 
\gamma= \displaystyle \sum_{j=1}^{q} \biggl{(}1-\dfrac{1}{{\nu}_{j}} \biggr{)}> 3\,(=1+2). 
\end{equation}
If the Lagrangian Gauss map $\nu\colon \Sigma \to \C\cup \{\infty \}$ satisfies the property that 
all ${\alpha}_{j}$-points of $\nu$ have multiplicity at least ${\nu}_{j}$, then there exists a positive constant $C$, depending on 
${\alpha}_{1}, \ldots, {\alpha}_{q}$, but not the surface, such that for all $p\in \Sigma$ we have
$$
|K_{d{\tau}^{2}}(p)|^{1/2}\leq \dfrac{C}{d(p)}, 
$$
where $K_{d{\tau}^{2}}(p)$ is the Gaussian curvature of the metric $d{\tau}^{2}$ at $p$ and $d(p)$ is the geodesic distance from $p$ to 
the boundary of $\Sigma$. 
\end{theorem}

As a corollary of Theorem \ref{thm-4-improper1}, 
we have the following ramification theorem for the Lagrangian Gauss map of a weakly complete improper affine front in ${\R}^{3}$.

\begin{corollary}[{\cite[Theorem 3.2]{KN2012}}]\label{thm-4-improper2}
Let $\psi\colon \Sigma \to {\R}^{3}$ be a weakly complete improper affine front. 
Let $q\in \N$, ${\alpha}_{1}, \ldots, {\alpha}_{q}\in \C\cup \{\infty \}$ be distinct and ${\nu}_{1}, \ldots, {\nu}_{q}\in \N\cup \{\infty \}$. 
Suppose that the inequality (\ref{equ-4-imramification}) holds. If the Lagrangian Gauss map $\nu\colon \Sigma \to \C\cup \{\infty \}$ satisfies 
the property that all ${\alpha}_{j}$-points of $\nu$ have multiplicity at least ${\nu}_{j}$, then $\nu$ must be constant, that is, $\psi(\Sigma)$ is an 
elliptic paraboloid. In particular, if the Lagrangian Gauss map of a weakly complete improper affine front in ${\R}^{3}$ is nonconstant, 
then it can omit at most $3\,(=1+2)$ values. 
\end{corollary}

Since the singular points of $\psi$ correspond to the points where $|\nu|=1$, we can get a simple proof of the parametric affine Bernstein theorem 
(\cite{Ca1958}, \cite{Jo1954}) for an improper affine sphere from the viewpoint of function-theoretic properties of the Lagrangian Gauss map. 
For the details, see \cite[Corollary 3.6]{Ka2013}. 

As an application of Corollary \ref{thm-4-improper2}, we obtain the following analogue to the Ahlfors islands theorem. 

\begin{theorem}\label{thm-4-improper3}
Let $\psi\colon \Sigma \to {\R}^{3}$ be a weakly complete improper affine front. 
Let $q\in \N$, ${\alpha}_{1}, \ldots, {\alpha}_{q}\in \C\cup \{\infty \}$ be distinct, 
$D_{j}({\alpha}_{j}, \varepsilon):=\{z\in \C\cup \{\infty \} \,;\, |z, {\alpha}_{j}|< \varepsilon \}$ $(1\leq j\leq q)$ be pairwise disjoint and 
${\nu}_{1}, \ldots, {\nu}_{q}\in \N$. Suppose that the inequality (\ref{equ-4-imramification}) holds. 
Then there exists $\varepsilon > 0$ such that, if the Lagrangian Gauss map $\nu$ has no island of multiplicity less than ${\nu}_{j}$ over 
$D_{j}({\alpha}_{j}, \varepsilon)$ 
for all $j\in \{1, \ldots , q\}$, then $\nu$ must be constant, that is, $\psi(\Sigma)$ is an elliptic paraboloid. 
\end{theorem}

The important case of Theorem \ref{thm-4-improper3} is the case where $q=7\,(=2\cdot 1+5)$ and ${\nu}_{j}=2$ for each $j$ 
$(j=1,\ldots, q)$. 

\begin{corollary}\label{thm-4-improper4}
Let $\psi\colon \Sigma \to {\R}^{3}$ be a weakly complete improper affine front. 
Let ${\alpha}_{1}, \ldots, {\alpha}_{7} \in \C\cup \{\infty \}$ be distinct and 
$D_{j}({\alpha}_{j}, \varepsilon):=\{z\in \C\cup \{\infty \} \,;\, |z, {\alpha}_{j}|< \varepsilon \}$ $(1\leq j\leq 7)$. 
Then there exists $\varepsilon > 0$ such that, if the Lagrangian Gauss map $\nu$ has no simple island of over any of the small disks 
$D_{j}({\alpha}_{j}, \varepsilon)$ for all $j\in \{1, \ldots , 7 \}$, then $\nu$ must be constant, that is, $\psi(\Sigma)$ is an elliptic paraboloid. 
\end{corollary}

Finally, by applying Theorem \ref{thm-unicity}, we provide the following unicity theorem for the Lagrangian Gauss maps of weakly complete 
improper affine fronts in ${\R}^{3}$. 

\begin{theorem}\label{thm-4-improper5}
Let $\psi\colon \Sigma \to {\R}^{3}$ and $\widehat{\psi}\colon \widehat{\Sigma}\to {\R}^{3}$ be two improper affine fronts and 
assume that there exists a conformal diffeomorphism $\Psi\colon \Sigma \to \widehat{\Sigma}$. Let $\nu\colon \Sigma \to \C\cup\{\infty \}$ and 
$\hat{\nu}\colon \widehat{\Sigma}\to \C\cup\{\infty \}$ be the Lagrangian Gauss maps of $\psi(\Sigma)$ and $\widehat{\psi}(\widehat{\Sigma})$, respectively. 
Suppose that there exist $q$ distinct points ${\alpha}_{1}, \ldots, {\alpha}_{q}\in \C\cup \{\infty \}$ 
such that ${\nu}^{-1}({\alpha}_{j})=(\hat{\nu}\circ \Psi)^{-1}({\alpha}_{j})$ $(1\leq j\leq q)$. 
If $q \geq 6 \,(=(1+4)+1)$ and either $\psi (\Sigma)$ or $\widehat{\psi} (\widehat{\Sigma})$ is weakly complete, then either 
$\nu\equiv \hat{\nu}\circ \Psi$ or $\nu$ and $\hat{\nu}$ are both constant, that is, $\psi (\Sigma)$ and  $\widehat{\psi} (\widehat{\Sigma})$ 
are both elliptic paraboloids. 
\end{theorem}

\subsection{Ratio of canonical forms of flat fronts in ${\H}^{3}$} 

For a holomorphic Legendrian immersion $\Lc\colon \Sigma \to SL(2, \C)$ on a simply connected Riemann surface $\Sigma$, the projection 
$$
f:= \Lc{\Lc}^{\ast}\colon \Sigma \to {\H}^{3}
$$
gives a {\it flat front} in ${\H}^{3}$. Here, flat fronts in ${\H}^{3}$ are flat surfaces in ${\H}^{3}$ with some admissible singularities  
(see \cite{KRUY2007}, \cite{KUY2004} for the definition of flat fronts in ${\H}^{3}$). We call $\Lc$ the {\it holomorphic lift} of $f$. 
Since $\Lc$ is a holomorphic Legendrian map, ${\Lc}^{-1}{d\Lc}$ is off-diagonal (see \cite{GMM2000}, \cite{KUY2003}, \cite{KUY2004}). 
If we set that 
$$
{\Lc}^{-1}{d\Lc} = \left(
\begin{array}{cc}
0      & \theta  \\
\omega & 0
\end{array}
\right), 
$$
then the pull-back of the canonical Hermitian metric of $SL(2, \C)$ by $\Lc$ is represented as 
$$
ds^{2}_{\Lc}:=|\omega|^{2}+|\theta|^{2}
$$ 
for holomorphic $1$-forms $\omega$ and $\theta$ on $\Sigma$. A flat front $f$ is said to be {\it weakly complete} 
if the metric $ds^{2}_{\Lc}$ is complete (\cite{KRUY2009, UY2011}). We define a meromorphic function on $\Sigma$ by the ratio of canonical forms 
$$
\rho := \dfrac{\theta}{\omega}. 
$$
Then a point $p\in \Sigma$ is a singular point of $f$ if and only if $|\rho (p)|=1$ (\cite{KRSUY2005}). We remark that 
$$
ds^{2}_{\Lc}=|\omega|^{2}+|\theta|^{2}=(1+|\rho|^{2})|\omega|^{2}. 
$$

Applying Theorem \ref{thm-ramification} to the metric $ds^{2}_{\Lc}$, we can get the following curvature estimate. 
This is a generalization of \cite[Theorem 4.8]{Ka2013}. 

\begin{theorem}\label{thm-4-flat1}
Let $f\colon \Sigma \to {\H}^{3}$ be a flat front on a simply connected Riemann surface $\Sigma$. 
Let $q\in \N$, ${\alpha}_{1}, \ldots, {\alpha}_{q}\in \C\cup \{\infty \}$ be distinct and ${\nu}_{1}, \ldots, {\nu}_{q}\in \N\cup \{\infty \}$. 
Suppose that the inequality (\ref{equ-4-imramification}) holds.  
If the ratio of canonical forms $\rho\colon \Sigma \to \C\cup \{\infty \}$ satisfies the property that 
all ${\alpha}_{j}$-points of $\rho$ have multiplicity at least ${\nu}_{j}$, then there exists a positive constant $C$, depending on 
${\alpha}_{1}, \ldots, {\alpha}_{q}$, but not the surface, such that for all $p\in \Sigma$ we have
$$
|K_{ds^{2}_{\Lc}}(p)|^{1/2}\leq \dfrac{C}{d(p)}, 
$$
where $K_{ds^{2}_{\Lc}}(p)$ is the Gaussian curvature of the metric $ds^{2}_{\Lc}$ at $p$ and $d(p)$ is the geodesic distance from $p$ to 
the boundary of $\Sigma$. 
\end{theorem}

If $\Sigma$ is not simply connected, then we consider that $\rho$ is a meromorphic function on its universal covering surface $\widetilde{\Sigma}$. 
As a corollary of Theorem \ref{thm-4-flat1}, 
we have the following ramification theorem for the ratio of canonical forms of a weakly complete improper affine front in ${\H}^{3}$.

\begin{corollary}[{\cite[Theorem 3.2]{Ka2013-2}}]\label{thm-4-flat2}
Let $f\colon \Sigma \to {\H}^{3}$ be a weakly complete flat front. 
Let $q\in \N$, ${\alpha}_{1}, \ldots, {\alpha}_{q}\in \C\cup \{\infty \}$ be distinct and ${\nu}_{1}, \ldots, {\nu}_{q}\in \N\cup \{\infty \}$. 
Suppose that the inequality (\ref{equ-4-imramification}) holds. If the ratio of canonical forms $\rho$ satisfies 
the property that all ${\alpha}_{j}$-points of $\rho$ have multiplicity at least ${\nu}_{j}$, then $\rho$ must be constant, that is, 
$f(\Sigma)$ is a horosphere or a hyperbolic cylinder. In particular, if the ratio of canonical forms of a weakly complete flat front in ${\H}^{3}$ is nonconstant, 
then it can omit at most $3\,(=1+2)$ values. 
\end{corollary}

As an application of Corollary \ref{thm-4-flat2}, we can obtain a simple proof of the classification (\cite{Sa1973}, \cite{VV1971}) of 
complete nonsingular flat surfaces in ${\H}^{3}$. For the details, see \cite[Corollary 4.5]{Ka2013-2}. 

As another application of Corollary \ref{thm-4-flat2}, we get the following analogue to the Ahlfors islands theorem. 

\begin{theorem}[{\cite[Corollary 4.2]{Ka2013-2}}]\label{thm-4-flat3}
Let $f\colon \Sigma \to {\H}^{3}$ be a weakly complete flat front. 
Let $q\in \N$, ${\alpha}_{1}, \ldots, {\alpha}_{q}\in \C\cup \{\infty \}$ be distinct, 
$D_{j}({\alpha}_{j}, \varepsilon):=\{z\in \C\cup \{\infty \} \,;\, |z, {\alpha}_{j}|< \varepsilon \}$ $(1\leq j\leq q)$ be pairwise disjoint and 
${\nu}_{1}, \ldots, {\nu}_{q}\in \N$. Suppose that the inequality (\ref{equ-4-imramification}) holds. 
Then there exists $\varepsilon > 0$ such that, if the ratio of canonical forms $\rho$ has no island of multiplicity less than ${\nu}_{j}$ over 
$D_{j}({\alpha}_{j}, \varepsilon)$ 
for all $j\in \{1, \ldots , q\}$, then $\rho$ must be constant, that is, $f(\Sigma)$ is a horosphere or a hyperbolic cylinder. 
\end{theorem}

The important case of Theorem \ref{thm-4-flat3} is the case where $q=7\,(=2\cdot 1+5)$ and ${\nu}_{j}=2$ for each $j$ 
$(j=1,\ldots, q)$. 

\begin{corollary}[{\cite[Corollary 4.3]{Ka2013-2}}]\label{thm-4-flat4}
Let $f\colon \Sigma \to {\H}^{3}$ be a weakly complete flat front. 
Let ${\alpha}_{1}, \ldots, {\alpha}_{7} \in \C\cup \{\infty \}$ be distinct and 
$D_{j}({\alpha}_{j}, \varepsilon):=\{z\in \C\cup \{\infty \} \,;\, |z, {\alpha}_{j}|< \varepsilon \}$ $(1\leq j\leq 7)$. 
Then there exists $\varepsilon > 0$ such that, if the ratio of canonical forms $\rho$ has no simple island of over any of the small disks 
$D_{j}({\alpha}_{j}, \varepsilon)$ for all $j\in \{1, \ldots , 7 \}$, then $\nu$ must be constant, that is, $f(\Sigma)$ is a horosphere or a hyperbolic 
cylinder. 
\end{corollary}

Finally, by applying Theorem \ref{thm-unicity}, we provide the following unicity theorem for the ratios of canonical forms of weakly complete 
flat fronts in ${\H}^{3}$. 

\begin{theorem}\label{thm-4-flat5}
Let $f\colon \Sigma \to {\H}^{3}$ and $\widehat{f}\colon \widehat{\Sigma}\to {\R}^{3}$ be two flat fronts on simply connected Riemann surfaces and 
assume that there exists a conformal diffeomorphism $\Psi\colon \Sigma \to \widehat{\Sigma}$. Let $\rho\colon \Sigma \to \C\cup\{\infty \}$ and 
$\hat{\rho}\colon \widehat{\Sigma}\to \C\cup\{\infty \}$ be the ratio of canonical forms $f(\Sigma)$ and $\widehat{f}(\widehat{\Sigma})$, respectively. 
Suppose that there exist $q$ distinct points ${\alpha}_{1}, \ldots, {\alpha}_{q}\in \C\cup \{\infty \}$ 
such that ${\rho}^{-1}({\alpha}_{j})=(\hat{\rho}\circ \Psi)^{-1}({\alpha}_{j})$ $(1\leq j\leq q)$. 
If $q \geq 6 \,(=(1+4)+1)$ and either $f (\Sigma)$ or $\widehat{f} (\widehat{\Sigma})$ is weakly complete, then either 
$\rho\equiv \hat{\rho}\circ \Psi$ or $\rho$ and $\hat{\rho}$ are both constant. 
\end{theorem}


\end{document}